\documentclass[en,oneside,bbfont,article]{amsart} 
\usepackage[lmargin = 30mm, rmargin = 30mm, bmargin = 25mm, tmargin=25mm]{geometry}
\usepackage[utf8]{inputenc}

\usepackage{mathtools}

\providecommand{\examplename}{Example}

\newtheorem{theorem}{Theorem}[section]
\newtheorem{proposition}[theorem]{Proposition}
\newtheorem{corollary}[theorem]{Corollary}
\newtheorem{lemma}[theorem]{Lemma}
\theoremstyle{remark}
\newtheorem{remarkx}[theorem]{Remark}
\newenvironment{remark}
    {\pushQED{\qed}\remarkx}
    {\popQED\endremarkx}
\theoremstyle{definition}
\newtheorem*{example*}{\protect\examplename}
\newtheorem{examplex}[theorem]{\protect\examplename}
\newenvironment{example}
    {\pushQED{\qed}\examplex}
    {\popQED\endexamplex}
\theoremstyle{definition}

\newtheorem*{assumption*}{Assumption}

\usepackage{mathrsfs}
\usepackage{amssymb}
\usepackage{amstext}
\usepackage{dsfont}
\usepackage{amsthm}
\usepackage{multicol}
\usepackage{amsmath}
\usepackage{bm}
\usepackage{comment}
\usepackage{caption,subcaption,enumerate,enumitem}
\usepackage{scalefnt}
\usepackage{float,graphicx,verbatim}
\usepackage{tikz}
\usetikzlibrary{calc,shapes,arrows,positioning,decorations.pathreplacing,calligraphy}
\usepackage{pgfplots}
\pgfplotsset{compat=1.16}
\RequirePackage[colorlinks=true,
	citecolor=blue,urlcolor=blue]{hyperref}

\usepackage[backend=biber,style=numeric,citestyle=numeric,maxbibnames=99,sortcites=true]{biblatex}
\renewbibmacro{in:}{}
\appto{\bibsetup}{\sloppy}

\newcommand\N{\mathbb{N}}

\newcommand{\mS}{\mathcal{S}}
\newcommand{\mZ}{\mathcal{Z}}
\newcommand\R{\mathbb{R}}

\newcommand\E{\mathds{E}}
\newcommand\p{\mathds{P}}
\newcommand\1{\mathds{1}}
\newcommand\Oh{\mathcal{O}}

\newcommand{\mL}{\mathcal{L}}

\newcommand\da{\downarrow}
\newcommand\ua{\uparrow}
\newcommand\ra{\rightarrow}

\newcommand{\ve}{\varepsilon}

\newcommand\cid{\xrightarrow{d}}

\newcommand\Var{\mbox{Var}}

\newcommand\nf[1]{\normalfont{#1}}

\newcommand{\D}{\mathrm{d}}
\newcommand{\ov}[1]{\overline{#1}}
\newcommand{\un}[1]{\underline{#1}}
\newcommand{\wt}[1]{\widetilde{#1}}
\newcommand{\wh}[1]{\widehat{#1}}

\newcommand{\BG}{\mathrm{BG}}

\newcommand\jorge[1]{{\color{red}#1}}
\newcommand\david[1]{{\color{blue}#1}}

\title{How smooth can the convex hull of a L\'evy path be?}
\date{\today}
\author{David Bang, Jorge Gonz\'alez C\'azares \& Aleksandar Mijatovi\'c
}

\address{Department of Statistics, University of Warwick, \and The Alan Turing Institute, UK}

\email{david.bang@warwick.ac.uk}
\email{jorge.i.gonzalez-cazares@warwick.ac.uk}
\email{a.mijatovic@warwick.ac.uk}

\begin{document}

\begin{abstract}
We describe the rate of growth of the derivative $C'$ of the convex minorant of a L\'evy path at times where $C'$ increases continuously. Since the convex minorant is piecewise linear, $C'$ may exhibit such behaviour either at the vertex time $\tau_s$ of finite slope $s=C'_{\tau_s}$ or at time $0$ where the slope is $-\infty$. 
While the convex hull depends on the entire path, we show that the local fluctuations of the derivative $C'$  depend only on the fine structure of the small jumps of the L\'evy process and are the same for all time horizons. In the domain of attraction of a stable process, we establish sharp results essentially characterising the modulus of continuity of $C'$  
up to sub-logarithmic factors. As a corollary 
we obtain novel results for the growth rate at $0$ of meanders
in a wide class of L\'evy processes.
\end{abstract}



\maketitle

\section{Introduction}
\label{sec:intro}

The class of L\'evy processes with paths whose graphs have convex hulls in the plane with smooth boundary almost surely has recently been characterised  in~\cite{SmoothCM}. In fact, as explained in~\cite{SmoothCM}, to understand whether the boundary is smooth at a point with tangent of a given slope, it suffices to analyse whether the right-derivative $C'=(C'_t)_{t\in(0,T)}$ of the convex minorant $C=(C_t)_{t\in[0,T]}$ of a L\'evy process $X=(X_t)_{t\in[0,T]}$ is continuous as it attains that slope (recall that $C$ is the pointwise largest convex function satisfying $C_t\leq X_t$ for all $t\in[0,T]$). The main objective of this paper is to quantify the smoothness of the boundary of the convex hull of $X$ by quantifying the modulus of continuity of $C'$ via its lower and upper functions. In the case of times $0$ and $T$, we quantify the degree of smoothness of the boundary of the convex hull by analysing the rate at which $|C'_t|\to\infty$ as $t$ approaches either $0$ or $T$ (see \href{https://youtu.be/9uCge3eMHQg}{YouTube}~\cite{Presentation_AM} for a short presentation of our results).

It is known that $C$ is a piecewise linear convex function~\cite{MR2978134,fluctuation_levy} and the image of the right-derivative $C'$ over the open intervals of linearity of $C$ is a countable random set $\mS$ with a.s. deterministic limit points that do not depend on the time horizon $T$, see~\cite[Thm~1.1]{SmoothCM}. These limit points of $\mS$ determine the continuity of $C'$ on $(0,T)$ outside of the open intervals of constancy of $C'$, see~\cite[App.~A]{SmoothCM}. Indeed, the \textit{vertex time process}~$\tau=(\tau_s)_{s\in\R}$, given by $\tau_s\coloneqq \inf\{t\in(0,T):C'_t>s\}\wedge T$ (where $a\wedge b\coloneqq\min\{a,b\}$ and $\inf\emptyset\coloneqq\infty$), is the right-inverse of the non-decreasing process $C'$. The process $\tau$ finds the times in $[0,T]$ of the vertices of the convex minorant $C$ (see~\cite[Sec.~2.3]{fluctuation_levy}), so the only possible discontinuities of $C'$ lie in the range of $\tau$. Clearly, it suffices to analyse only the times $\tau_s$ for which $C'$ is non-constant on the interval $[\tau_s,\tau_s+\ve)$ for every $\ve>0$ (otherwise, $\tau_s$ is the time of a vertex isolated from the right). At such a time, the continuity of $C'$ can be described in terms of a limit set of $\mS$. In the present paper we analyse the quality of the right-continuity of $C'$ at such points. By time reversal, analogous results apply for the left-continuity of $t\mapsto C'_t$ on $(0,T)$ (i.e., as $t\ua\tau_s$ for $s\in\R$) and for the explosion of $C'_t$ as $t\ua T$. Throughout the paper, the variable $s\in\R$ will be reserved for \emph{slope}, indexing the vertex time process $\tau$.

\subsection{Contributions}

We describe the small-time fluctuations of the derivative of the boundary of the convex hull of $X$ at its points of smoothness. This requires studying the local growth of $C'$ in two regimes: at \emph{finite slope} (FS) $s$ in the deterministic set $\mL^+(\mS)\subset\R$ of right-limit points\footnote{A point $x$ is a right-limit point of $A\subset\R$, denoted $x\in\mL^+(A)$ if $A\cap(x,x+\ve)\ne\emptyset$ for all $\ve>0$ (see also~\cite[App.~A]{SmoothCM}).} of the set of slopes~$\mS$ and at \emph{infinite slope} (IS) for L\'evy processes of infinite variation, see Figure~\ref{fig:CM_0_and_postmin} below. In terms of times, regime (FS) with $s\in\mL^+(\mS)$ analyses how $C'$ leaves the slope~$s$ at vertex time $\tau_s$ in $[0,T)$ and regime (IS) analyses how $C'$ enters from $-\infty$ at time $0=\lim_{u\da-\infty}\tau_u$. At all other times $t\in(0,T)\setminus\{\tau_s\,:\,s\in\mL^+(\mS)\}$, the derivative $C'$ is constant on $[t,t+\ve)$ for some sufficiently small $\ve>0$. In particular, in what follows we exclude all L\'evy processes that are compound Poisson with drift, since $C'$ only takes finitely many values in that case. 

\begin{figure}[ht]
\centering
\includegraphics[width=.49\textwidth]{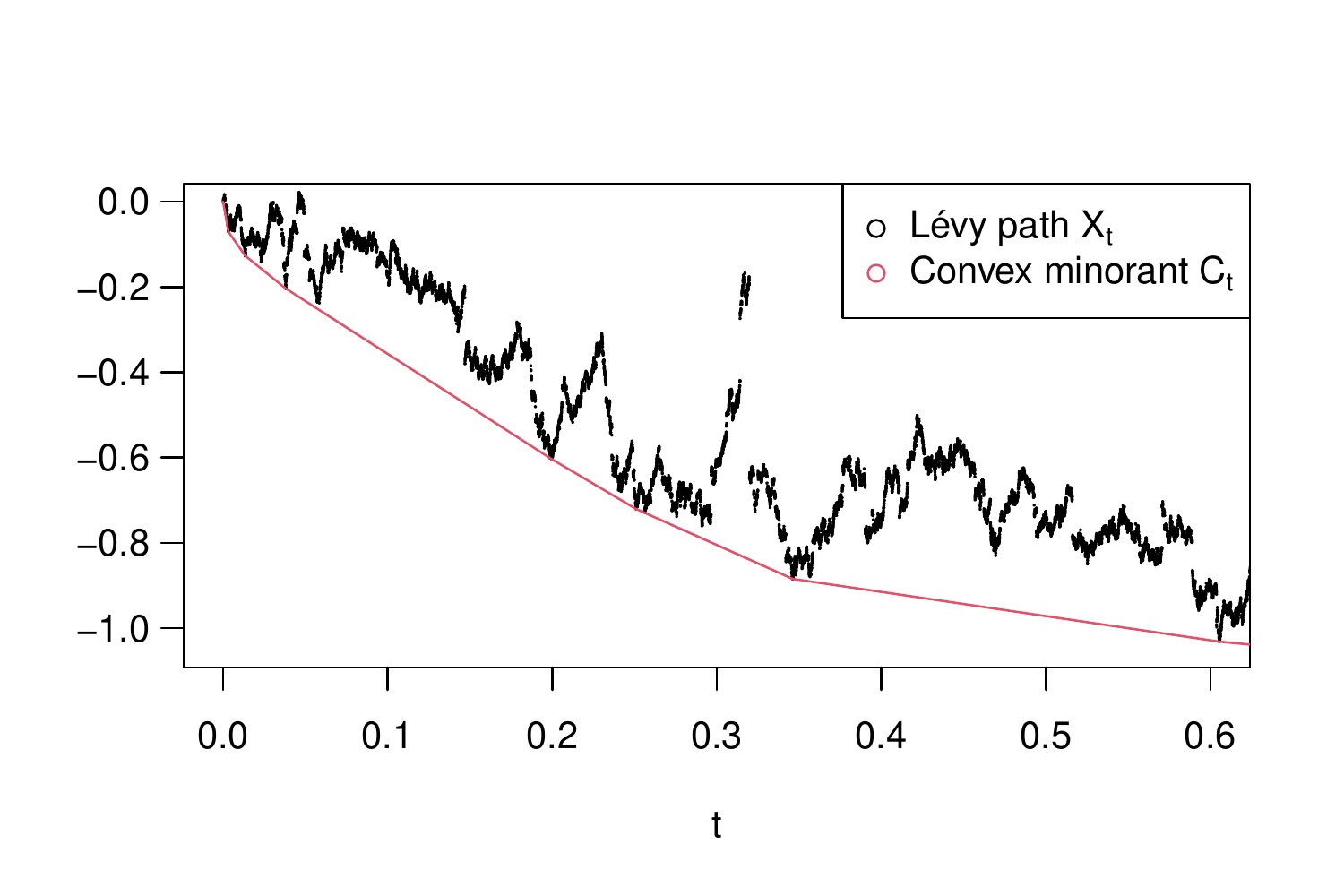}
\includegraphics[width=.49\textwidth]{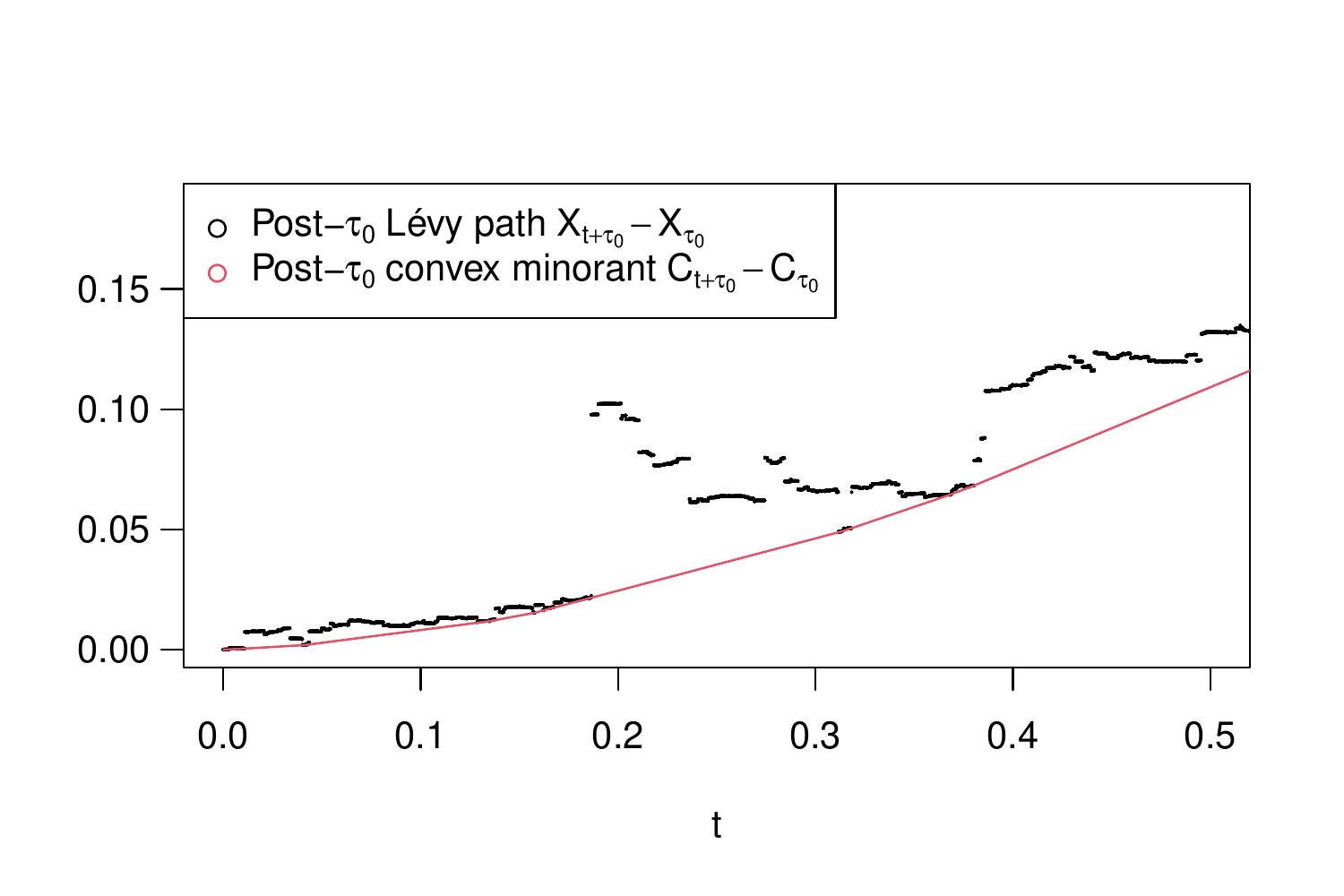}
\caption{\small The picture on the left shows the path of an $\alpha$-stable L\'evy process $X$ with $\alpha\in (1,2)$ and its convex minorant $C$ starting at time $0$. The picture on the right shows the post-minimum process $(X_{t+\tau_0}-X_{\tau_0})_{t\in[0,T-\tau_0]}$ of an $\alpha$-stable process with $\alpha \in (0,1)$ and its corresponding convex minorant $(C_{t+\tau_0}-C_{\tau_0})_{t\in[0,T-\tau_0]}$. Note that, in the case $\alpha\in(0,1)$, the derivative $C'$ is continuous only at $\tau_0$, i.e. at $t=0$ in the graph, and at no other contact point between the path and its convex minorant.}
\label{fig:CM_0_and_postmin}
\end{figure}

{\textbf{Regime (FS): $C'$ immediately after $\tau_s$.}} 
Given a slope $s\in\R$, we have $s\notin\mS$ a.s. by~\cite[Thm~3.1]{fluctuation_levy} since the law of $X$ is diffuse.
By~\cite[Thm~1.1]{SmoothCM}, $s\in\mL^+(\mS)$ if and only if the derivative $C'$ attains level $s$ at a unique time $\tau_s\in(0,T)$ (i.e. $C'_{\tau_s}=s$) and is not constant on every interval $[\tau_s,\tau_s+\ve)$, $\ve>0$, a.s. Moreover, $s\in\mL^+(\mS)$ if and only if $\int_0^1 \p(X_t/t\in(s,s+\ve))t^{-1}\D t=\infty$ for all $\ve>0$. The regime (FS) includes an infinite variation process $X$ if it is strongly eroded (implying $\mL^+(\mS)=\R$) or, more generally, if $(X_t-st)_{t\ge 0}$ is eroded (implying $s\in\mL^+(\mS)$), see~\cite{SmoothCM}. Moreover, regime (FS) includes a finite variation process $X$ at slope $s\in\mL^+(\mS)$ if and only if the natural drift $\gamma_0=\lim_{t\da 0}X_t/t$ equals $s$ and $\int_0^1 \p(X_t>\gamma_0t)t^{-1}\D t=\infty$ or, equivalently, if the positive half-line is regular for $(X_t-\gamma_0 t)_{t\ge 0}$ (see~\cite[Cor.~1.4]{SmoothCM} for a characterisation in terms of the L\'evy measure of $X$ or its characteristic exponent). 

Our results in regime (FS) are summarised as follows. For any process with $s\in\mL^+(\mS)$, Theorem~\ref{thm:post-min-lower} establishes general sufficient conditions identifying when $\liminf_{t\da 0}(C'_{t+\tau_s}-s)/f(t)$ is either $0$ a.s. or $\infty$ a.s. In particular, we show that $\liminf_{t\da 0}(C'_{t+\tau_s}-s)/f(t)$ cannot take a positive finite value if $X$ has jumps of both signs and is an $\alpha$-stable with $\alpha\in(0,1]$ (recall that, if $\alpha>1$, then $\mL^+(\mS)=\emptyset$ by~\cite[Prop.~1.6]{SmoothCM}). 

For processes $X$ in the small-time domain of attraction of an $\alpha$-stable process with $\alpha \in (0,1)$ (see Subsection~\ref{subsec:upper-post-min} below for definition), Theorem~\ref{thm:upper_fun_C'_post_min} finds a parametric family of functions $f$ that essentially determine the upper fluctuations of $C'_{t+\tau_{s}}-s$ up to sublogarithmic factors. In particular, Theorem~\ref{thm:upper_fun_C'_post_min} determines when $\limsup_{t\da 0}(C'_{t+\tau_s}-s)/f(t)$ equals $0$ a.s. or $\infty$ a.s., essentially characterising the right-modulus of continuity\footnote{We say that a non-decreasing function $\varphi:[0,\infty)\to[0,\infty)$ is a right-modulus of continuity of a right-continuous function $g$ at $x\in\R$ if $\limsup_{y\da x}|g(y)-g(x)|/\varphi(y-x)<\infty$.} of $C'$ at $\tau_s$. The family of functions $f$ is given in terms of the regularly varying normalising function of $X$. 


{\textbf{Regime (IS): $C'$ immediately after $0$.}} 
The boundary of the convex hull of $X$ is smooth at the origin if and only if $\lim_{t\da 0}C'_t=-\infty$ a.s., which is equivalent to $X$ being of infinite variation (see~\cite[Prop.~1.5~\&~Sec.~1.1.2]{SmoothCM}). If $X$ has finite variation, then $C'$ is bounded (see~\cite[Prop.~1.3]{SmoothCM}). In this case, $C'$ has positive probability of being non-constant on the interval $[0,\ve)$ for every $\ve>0$ if and only if the negative half-line is not regular. Moreover, if this event occurs, then $C'_t$ approaches the natural drift $\gamma_0$ as $t\da 0$ by~\cite[Prop.~1.3(b)]{SmoothCM} and the local behaviour of $C'$ at $0$ would be described by the results of regime (FS). Thus, in regime (IS) we only consider L\'evy processes of infinite variation. 

Our results in regime (IS) are summarised as follows. For any infinite variation process $X$, Theorem~\ref{thm:C'_limsup} establishes general sufficient conditions for $\limsup_{t \da 0}|C'_t|f(t)$ to equal either $0$ a.s. or $\infty$ a.s. In particular, we show that $\limsup_{t\da 0}|C'_t|f(t)$ cannot take a positive finite value if $X$ is $\alpha$-stable with $\alpha\in[1,2)$ and has (at least some) negative jumps. 

If the L\'evy process lies in the domain of attraction of an $\alpha$-stable process, with $\alpha \in (1,2]$, Theorem~\ref{thm:lower_fun_C'} finds a parametric family of functions $f$ that essentially determine the lower fluctuations of $C'$ up to sublogarithmic functions. The function $f$ is given in terms of the regularly varying normalising function of $X$. Again, these results describe the right-modulus of continuity of the derivative of the boundary of the convex hull of $X$ (as a closed curve in $\R^2$) at the origin. In this case, for a sufficiently small $\ve>0$, we may locally parametrise the curve $((t,C_t);t\in[0,\ve])$, as $((\varsigma(t),t);t\in[C_\ve,0])$, using a local inverse $\varsigma(t)$ of $C_t$ with left-derivative $\varsigma'(t)=1/C'_{\varsigma(t)}$ that vanishes at $0$ (since $\lim_{t\da 0}1/|C'_t|=0$ a.s.). Thus, the left-modulus of continuity of $\varsigma$ at $0$ is described by the upper and lower limits of $(|C'_t|f(t))^{-1}$ as $t\da 0$, the main focus of our results in this regime.

{\textbf{Consequences for the path of a L\'evy process and its meander.}}
In Subsection~\ref{subsec:applications} we present some implications the results in this paper have for the path of $X$. We find that, under certain conditions, the local fluctuations of $X$ can be described in terms of those of $C'$, yielding novel results for the local growth of the post-minimum process of $X$ and the corresponding L\'evy meander (see Lemma~\ref{lem:upper_fun_Lev_path_post_slope} and Corollaries~\ref{cor:post_tau_s_Levy_path} and~\ref{cor:post-tau_s-Levy-path-attraction} below). 

\subsection{Strategy and ideas behind the proofs}
An overview of the proofs of our results is as follows. First we show that, under our assumptions, the local properties of $C'$ do not depend on the time horizon $T$. This reduces the problem to the case where the time horizon $T$ is independent of $X$ and exponentially distributed (the corresponding right-derivative is denoted $\wh C'$). Second, we translate the problem of studying the local behaviour of $\wh C'$ to the problem of studying the local behaviour of its inverse: the vertex time process $\wh\tau$. Third, we exploit the fact that, since the time horizon $T$ is an independent exponential random variable with mean $1/\lambda$, the vertex time process $\wh\tau$ is a time-inhomogenous non-decreasing additive process (i.e., a process with independent but non-stationary increments) and its Laplace exponent is given by (see~\cite[Thm~2.9]{fluctuation_levy}):
\begin{equation}
\label{eq:cf_tau0}
\E[e^{-w\wh\tau_u}]
=e^{-\Phi_u(w)}, \quad \text{where} \quad \Phi_u(w)\coloneqq \int_0^\infty (1-e^{-w t})e^{-\lambda t}\p(X_t\le ut)\frac{\D t}{t},
\quad\text{for $w\ge 0$, $u\in\R$.}
\end{equation}
These three observations reduce the problem to the analysis of the fluctuations of the additive process~$\wh\tau$. 

The local properties of $C'$ are entirely driven by the small jumps of $X$. However, different facets of the small-jump activity of $X$ dominate in each regime, resulting in related but distinct results and criteria. Indeed, regime (FS) corresponds to the short-term behaviour of $\wh\tau_{s+u}-\wh\tau_{s}$ as $u\da 0$ while regime (IS) corresponds to the long-term behaviour of $\wh\tau_u$ as $u\to -\infty$ (note that, when $X$ is of infinite variation, $\tau_u>0$ for $u\in\R$ and $\lim_{u\to -\infty}\wh\tau_u=0$ a.s.). This bears out in a difference in the behaviour of the Laplace exponent $\Phi$ of $\wh\tau$ at either bounded or unbounded slopes and leads to an interesting diagonal connection in behaviour that we now explain.

Our main tool is the novel description of the upper and lower fluctuations of a non-decreasing time-inhomogenous additive process $Y$ started at $Y_0=0$, in terms of its time-dependent L\'evy measure and Laplace exponent. In our applications, the process $Y$ is given by $(\wh\tau_{u+s}-\wh\tau_s)_{u\ge 0}$ in regime (FS) and $(\wh\tau_{-1/u})_{u\ge 0}$ (with conventions $-1/0=-\infty$ and $\wh\tau_{-\infty}=0$) in regime (IS). Then our main technical tools, Theorems~\ref{thm:limsup_L} and~\ref{thm:Y_limsup} of Section~\ref{sec:additive} below, describing the upper and lower fluctuations of $Y$, also serve to describe the lower and upper fluctuations, respectively, of the right-inverse $L$ of $Y$.
Since, in regime (FS), we have $\wh C'_{t+\tau_s}-s=L_t$ but, in regime (IS), we have $\wh C'_t=-1/L_t$, the lower (resp. upper) fluctuations of $\wh C'$ in regime (FS) will have a similar structure to the upper (resp. lower) fluctuations of $\wh C'$ in regime (IS). This diagonal connection is \emph{a priori} surprising as the processes considered by either regime need not have a clear connection to each other. Indeed, regime (FS) considers most finite variation processes and only some infinite variation processes while regime (IS) considers exclusively infinite variation processes. This diagonal connection is reminiscent of the duality between stable process with stability index $\alpha\in(1,2]$ and a corresponding stable process with stability index $1/\alpha\in[1/2,1)$ arising in the famous time-space inversion first observed by Zolotarev for the marginals and later studied by Fourati~\cite{MR2218871} for the ascending ladder process (see also~\cite{MR4237257} for further extensions of this duality). 

The lower and upper fluctuations of the corresponding process $Y$ require varying degrees of control on its Laplace exponent $\Phi$ in~\eqref{eq:cf_tau0}. The assumptions of Theorem~\ref{thm:limsup_L} require tight two-sided estimates of $\Phi$, not needed in Theorem~\ref{thm:Y_limsup}. When applying Theorem~\ref{thm:limsup_L}, we are compelled to assume $X$ lies in the domain of attraction of an $\alpha$-stable process. In regime (FS) this assumption yields sharp estimates on the density of $X_t$ as $t\da 0$, which in turn allows us to control the term $\p(0<X_t-st\le ut)$ for small $t>0$ in the Laplace exponent $\Phi_{s+u}-\Phi_s$ of $\wh \tau_{u+s}-\wh \tau_s$ as $u\da 0$, cf.~\eqref{eq:cf_tau0} above. The growth rate of the density of $X_t$ as $t\da 0$ is controlled is by \emph{lower} estimates on the small-jump activity of $X$ given in Lemma~\ref{lem:generalized_Picard} below, a refinement of the results in~\cite{picard_1997} for processes attracted to a stable process. In regime (IS) we require control over the negative tail probabilities $\p(X_t\le ut)$ for small $t>0$ appearing in the Laplace exponent $\Phi_u$ of $\wh \tau_u$ as $u\to-\infty$, cf.~\eqref{eq:cf_tau0}. The behaviour of these tails are controlled by \emph{upper} estimates of the small-jump activity of $X$, which are generally easier to obtain. In this case, moment bounds for the small-jump component of the L\'evy process and the convergence in Kolmogorov distance implied by the attraction to the stable law, give sufficient control over these tail probabilities. 

\subsection{Connections with the literature}


In~\cite{MR1747095}, Bertoin finds the law of the convex minorant of Cauchy process on $[0,1]$ and finds the exact asymptotic behaviour (in the form of a law of interated logarithm with a positive finite limit) for the derivative $C'$ at times $0$, $1$ and any $\tau_s$, $s\in\R$. The methods in~\cite{MR1747095} are specific to Cauchy process with its linear scaling property, making the approach hard to generalise. In fact, the results in~\cite{MR1747095} are a direct consequence of the fact that the vertex time process $\wh \tau$ has a Laplace transform $\Phi$ in~\eqref{eq:cf_tau0} that factorises as $\Phi_u(w)=\p(X_1\le u)\Phi_{\infty}(w)$, making $\wh \tau$ a gamma subordinator under the deterministic time-change $u\mapsto\p(X_1\le u)$, cf.~Example~\ref{ex:Cauchy} below.

Paul L\'evy showed  that the boundary of the convex hull of a planar Brownian motion has no corners at any point, see~\cite{MR0029120}, motivating~\cite{MR972777} to characterise the modulus of continuity of the derivative of that boundary. 
Given the recent characterisation of the smoothness of the convex hull of a L\'evy path~\cite{SmoothCM}, the results in the present paper are likewise motivated by the study of the  modulus of continuity of the derivative of the boundary in this context.

The literature on the growth rate of the path of a L\'evy process $X$ is vast, particularly for subordinators, see e.g.~\cite{MR0002054,MR1746300,MR292163,MR210190,MR968135,MR2480786,MR2591911}. The authors in~\cite{MR292163,MR210190} study the growth rate of a subordinator at $0$ and $\infty$. In~\cite{MR210190} (see also~\cite[Prop~4.4]{MR1746300}) Fristedt fully characterises the upper fluctuations of a subordinator in terms of its L\'evy measure, a result we generalise in Theorem~\ref{thm:Y_limsup} to processes that need not have stationary increments. In~\cite[Thm~4.1]{MR1746300} (see also~\cite[Thm~1]{MR292163}, a function essentially characterising the exact lower fluctuations of a subordinator is constructed in terms of its Laplace exponent. These methods are not easily generalised to the time-inhomogenous case since the Laplace exponent is now bivariate and there is neither a one-parameter lower function to propose nor a clear extension to the proofs. 

In~\cite{MR1113220}, Sato establishes results for time-inhomogeneous non-decreasing additive processes similar to our result in Section~\ref{sec:additive}. The assumptions in~\cite{MR1113220} are given in terms of the transition probabilities of the additive process, which are generally intractable, particularly for the processes $(\wh\tau_{-1/u})_{u>0}$ and $(\wh\tau_{u+s}-\wh\tau_s)_{u\ge 0}$, considered here. Our results are also easier to apply in other situations as well, for example, to fractional Poisson processes (see definition in~\cite{MR3943682}).

The upper fluctuations of a L\'evy process at zero have been the topic of numerous studies, see~\cite{MR2370602,MR2591911} for the one-sided problem and~\cite{MR0002054,MR968135,MR2480786} for the two-sided problem. Similar questions have been considered for more general time-homogeneous Markov processes~\cite{Franziska_Kuhn,SoobinKimLeeLIL}. The time-homogeneity again plays an important role in these results. The lower fluctuations of a stochastic process is only qualitatively different from the upper fluctuations if the process is positive. This is the reason why this problem has mostly only been addressed for subordinators (see the references above) and for the running supremum of a L\'evy process, see e.g.~\cite{MR3019488}. We stress that the results in the present paper, while related in spirit to this literature, are fundamentally different in two ways. First, we study the \emph{derivative} of the convex minorant of a L\'evy path on $[0,T]$, which (unlike e.g. the running supremum) cannot be constructed locally from the restriction of the path of the L\'evy process to any short interval. Second, the convex minorant and its derivative are neither Markovian nor time-homogeneous. In fact, the only result in our context prior to our work is in the Cauchy case~\cite{MR1747095}, where the derivative of the convex minorant is an explicit gamma process under a deterministic time-change, cf.~Example~\ref{ex:Cauchy} below.


\subsection{Organisation of the article}
In Section~\ref{sec:small-time-derivative} we present the main results of this article. We split the section in four, according to regimes (FS) and (IS) and whether the upper or lower fluctuations of $C'$ are being described. The implications of the results in Section~\ref{sec:small-time-derivative} for the L\'evy process and meander are covered in Subsection~\ref{subsec:applications}. In Section~\ref{sec:additive}, technical results for general time-inhomogeneous non-decreasing additive processes are established. 
Section~\ref{sec:proofs} recalls from~\cite{fluctuation_levy} the definition and law of the vertex time process $\tau$  and provides the proofs of the results stated in Section~\ref{sec:small-time-derivative}. 
Section~\ref{sec:concluding_rem} concludes the paper.

\section{Growth rate of the derivative of the convex minorant}
\label{sec:small-time-derivative}

Let $X=(X_t)_{t \ge 0}$ be an infinite activity L\'evy process (see~\cite[Def.~1.6, Ch.~1]{MR3185174}). Let $C=(C_t)_{t \in [0,T]}$ be the convex minorant of $X$ on $[0,T]$ for some $T>0$. Put differently, $C$ is the largest convex function that is piecewise smaller than the path of $X$ (see~\cite[Sec.~3,p.~8]{fluctuation_levy}). In this section we analyse the growth rate of the right derivative of $C$, denoted by $C'=(C_t')_{t\in(0,T)}$, near time $0$ and at the vertex time $\tau_s=\inf\{t>0:C'_t>s\}\wedge T$ of the slope $s\in \R$ (i.e., the first time $C'$ attains slope $s$). More specifically, we give sufficient conditions to identify the values of the possibly infinite limits (for appropriate increasing functions $f$ with $f(0)=0$): 
$\limsup_{t\da 0}(C'_{t+\tau_s}-s)/f(t)$ \& $\liminf_{t\da 0}(C'_{t+\tau_s}-s)/f(t)$ in the finite slope (FS) regime and  $\limsup_{t\da 0}|C'_t|f(t)$ \& $\liminf_{t\da 0}|C'_t|f(t)$ in the infinite slope (IS) regime. The values of these limits are constants in $[0,\infty]$ a.s. by Corollary~\ref{cor:trivial} below. We note that these limits are invariant under certain modifications of the law of $X$, which we describe in the following remark.

\begin{remark}
\label{rem:modify_nu}\phantom{empty}
\begin{itemize}[leftmargin=2em, nosep]
\item[{\nf(a)}] Let $\p$ be the probability measure on the space where $X$ is defined. If the limits $\limsup_{t\da0}|C'_t|f(t)$, $\liminf_{t\da0}|C'_t|f(t)$, $\limsup_{t\da 0}(C'_{t+\tau_s}-s)/f(t)$ and $\liminf_{t\da 0}(C'_{t+\tau_s}-s)/f(t)$ are $\p$-a.s. constant, then they are also $\p'$-a.s. constant with the same value for any probability measure $\p'$ absolutely continuous with respect to $\p$. In particular, we may modify the L\'evy measure of $X$ on the complement of any neighborhood of $0$ without affecting these limits (see e.g.~\cite[Thm~33.1--33.2]{MR3185174}).
\item[{\nf(b)}] We may add a drift process to $X$ without affecting the limits at $0$ since such a drift would only shift $|C'_t|$ by a constant value and $f(t)\to 0$ as $t \downarrow0$. Similarly, for the limits of $(C'_{t+\tau_s}-s)/f(t)$ as $t\da 0$, it suffices to analyse the post-minimum process (i.e., the vertex time $\tau_0$) of the process $(X_t-st)_{t\ge 0}$. For ease of reference, our results are stated for a general slope $s$.\qedhere
\end{itemize}
\end{remark}

\subsection{Regime (FS): lower functions at time \texorpdfstring{$\tau_s$}{tau}}
\label{subsec:lower-post-min}

The following theorem describes the lower fluctuations of $C'_{t+\tau_s}-s$ as $t \da 0$. Recall that $\mL^+(\mS)$ is the a.s. deterministic set of right-limit points of the set of slopes $\mS$.

\begin{theorem}\label{thm:post-min-lower}
Let $s\in\mL^+(\mS)$ and $f$ be continuous and increasing, satisfying $f(t)\le 1=f(1)$ for $t\in(0,1]$ and $f(0)=0=\lim_{c\da0}\limsup_{t\da0}f(ct)/f(t)$. Let $c>0$ and consider the following conditions: 
\begin{gather}
\label{eq:post-min-Pi-large}
\int_0^1 \p(0<(X_t-st)/t \le f(t/c))\frac{\D t}{t}<\infty,\\
\label{eq:post-min-Pi-var}
\int_0^1\E\left[\frac{t}{f^{-1}((X_t-st)/t)^2}\1_{\{f(t/2)<(X_t-st)/t\le 1\}}\right]\D t<\infty,\\
\label{eq:post-min-Pi-mean}
2^n\int_0^{2^{-n}} \p(f(t/2)<(X_t-st)/t\le f(2^{-n}))\D t\to 0,
\quad\text{as }n\to\infty.
\end{gather}
Then the following statements hold.
\begin{itemize}[leftmargin=2.5em, nosep]
\item[{\nf{(i)}}] If \eqref{eq:post-min-Pi-large}--\eqref{eq:post-min-Pi-mean} hold for $c=1$, then $\liminf_{t\da0}(C'_{t+\tau_s}-s)/f(t)=\infty$ a.s.
\item[{\nf(ii)}] If~\eqref{eq:post-min-Pi-large} fails for every $c>0$, then $\liminf_{t\da0}(C'_{t+\tau_s}-s)/f(t)=0$ a.s.
\item[\nf{(iii)}] If $\liminf_{t\da0}(C'_{t+\tau_s}-s)/f(t)>1$ a.s., then \eqref{eq:post-min-Pi-large} holds for any $c>1$.
\end{itemize}
\end{theorem}

Some remarks are in order.

\begin{remark}
\label{rem:limsup_cond_post_min}\phantom{empty}
\begin{itemize}[leftmargin=2em, nosep]
\item[(a)] Any continuous regularly varying function $f$ of index $r>0$ satisfies the assumption in the theorem: $\lim_{c\da 0}\lim_{t\da 0}f(ct)/f(t)=\lim_{c\da 0}c^r=0$.
Moreover, the assumption $f(t)\le 1=f(1)$ for $t\in(0,1]$ is not necessary but makes  conditions~\eqref{eq:post-min-Pi-large}--\eqref{eq:post-min-Pi-mean} take a simpler form. 
\item[(b)] The proof of Theorem~\ref{thm:post-min-lower} is based on the analysis of the upper fluctuations of $\tau$ at slope $s$. Condition~\eqref{eq:post-min-Pi-large} ensures $(\tau_{u+s}-\tau_{s})_{u\ge 0}$ jumps finitely many times over the boundary $u\mapsto f^{-1}(u)$, condition~\eqref{eq:post-min-Pi-mean} makes the small-jump component of $(\tau_{u+s}-\tau_{s})_{u\ge 0}$  (i.e. the sum of the jumps at times $v\in[s,u+s]$ of size at most $f^{-1}(v)$) have a mean that tends to $0$ as $u\da 0$ and condition~\eqref{eq:post-min-Pi-var} controls the deviations of $(\tau_{u+s}-\tau_{s})_{u\ge 0}$ away from its mean.
\item[{(c)}] Note that~\eqref{eq:post-min-Pi-mean} holds if $\int_0^1\p(f(2^{-n}t/2)<(X_{2^{-n}t}-s2^{-n}t)/(2^{-n}t)\le f(2^{-n}))\D t\to 0$ as $n\to\infty$, which, by the dominated convergence theorem, holds if $\p(f(u/2)<(X_{u}-su)/u\le f(u/t))\to 0$ as $u\da 0$ for a.e. $t\in(0,1)$.
\item[{(d)}] Condition~\eqref{eq:post-min-Pi-var} in Theorem~\ref{thm:post-min-lower} requires access to the inverse $f^{-1}$ of the function $f$. In the special case when the function $f$ is concave, this assumption can be replaced with an assumption given in terms of $f$ (cf. Proposition~\ref{prop:Y_limsup} and Corollary~\ref{cor:L_liminf}). However, it is important to consider non-concave functions $f$, see Corollary~\ref{cor:power_func_liminf_post_min} below.\qedhere
\end{itemize}
\end{remark}

\subsubsection{Simple sufficient conditions for the assumptions of Theorem~\ref{thm:post-min-lower}
}\label{subsec:simp_suff_cond_tau_s} 

Let $f$ be as in Theorem~\ref{thm:post-min-lower}. By Theorem~\ref{thm:Y_limsup}(c) below (with the measure $\Pi(\D x,\D t)=\p((X_t-st)/t \in \D x)t^{-1}\D t$), the following condition implies~\eqref{eq:post-min-Pi-var}--\eqref{eq:post-min-Pi-mean}:
\begin{equation}
\label{eq:suff_low_post_min}
\int_0^1 \E\left[\frac{1}{f^{-1}((X_t-st)/t)}\1_{\{f(t/2)< (X_t-st)/t\le 1\}}\right]\D t <\infty.
\end{equation} 
If estimates on the density of $X_t$ are available (e.g., via assumptions on the generating triplet of $X$), \eqref{eq:suff_low_post_min} can be simplified further, see Corollary~\ref{cor:power_func_liminf_post_min} below. 

Throughout, we denote by $(\gamma,\sigma^2,\nu)$ the generating triplet of $X$ (corresponding to the cutoff function $x\mapsto\1_{(-1,1)}(x)$, see~\cite[Def.~8.2]{MR3185174}), where $\gamma \in \R$ is the drift parameter, $\sigma^2 \ge 0$ is the Gaussian coefficient and $\nu$ is the L\'evy measure of $X$ on $\R$. We also define the functions
\[
\ov\sigma^2(\ve)\coloneqq \sigma^2+\ov\sigma^2_+(\ve)+\ov\sigma^2_-(\ve),
\quad
\ov\sigma^2_+(\ve)\coloneqq\int_{(0,\ve)}x^2\nu(\D x),\quad
\ov\sigma^2_-(\ve)\coloneqq\int_{(-\ve,0)}x^2\nu(\D x),\quad
\text{for $\ve>0$.}
\]
Recall that, in regime (FS), we have $\sigma^2=0$ (see~\cite[Prop.~1.6]{SmoothCM}). Given two positive functions $g_1$ and $g_2$, we say $g_1(\ve)=\Oh(g_2(\ve))$ as $\ve\da 0$ if $\limsup_{\ve\da 0}g_1(\ve)/g_2(\ve)<\infty$. Similarly, we write $g_1(\ve)\approx g_2(\ve)$ as $\ve\da 0$ if $g_1(\ve)=\Oh(g_2(\ve))$ and $g_2(\ve)=\Oh(g_1(\ve))$.


\begin{corollary}
\label{cor:power_func_liminf_post_min}
Fix $\beta\in(0,1]$ and let $s\in\mL^+(\mS)$ and $f$ be as in Theorem~\ref{thm:post-min-lower}.
\begin{itemize}[leftmargin=2em, nosep]
\item[{\nf{(a)}}] If $\liminf_{\ve \da 0}\ve^{\beta-2}\ov\sigma^2(\ve)>0$, $f$ is differentiable with positive derivative $f'>0$ and the integrals $\int_0^1\int_{t/2}^1(f'(y)/y)t^{1-1/\beta}\D y\D t$ and $\int_0^1t^{-1/\beta}f(t)\D t$ are finite, then $\liminf_{t\da 0}(C'_{t+\tau_s}-s)/f(t)=\infty$ a.s.
\item[{\nf{(b)}}] Assume 
$\int_0^1 ((t^{-1/\beta}f(t))\wedge t^{-1})\D t=\infty$ and either of the following hold:
\begin{itemize}[leftmargin=2em, nosep]
\item[{\nf{(i)}}] $\ov\sigma^2(\ve)\approx\ve$ and $|\int_{(-1,1)\setminus(-\ve,\ve)}x\nu(\D x)|=\Oh(1)$ as $\ve\da 0$,
\item[{\nf{(ii)}}] $\beta\in(0,1)$ and $\ov\sigma^2_\pm(\ve)\approx\ve^{2-\beta}$ as $\ve\da 0$ for both signs of $\pm$,
\end{itemize}
then $\liminf_{t\da 0}(C'_{t+\tau_s}-s)/f(t)=0$ a.s. 
\end{itemize}
\end{corollary}

We stress that the sufficient conditions in Corollary~\ref{cor:power_func_liminf_post_min} are all in terms of the characteristics of the L\'evy process $X$ and the function $f$.

\begin{remark}
\label{rem:power_func_liminf_post_min}\phantom{empty}
\begin{itemize}[leftmargin=2em, nosep]
\item[{{(a)}}] The assumptions in Corollary~\ref{cor:power_func_liminf_post_min} are satisfied by most processes in the class $\mZ_{\alpha,\rho}$ of L\'evy processes in the small-time domain of attraction of an $\alpha$-stable distribution, see Subsection~\ref{subsec:upper-post-min} below (cf.~\cite[Eq.~(8)]{MR3784492}). Thus, the assumptions of part (a) in Corollary~\ref{cor:power_func_liminf_post_min} hold for any $X\in\mZ_{\alpha,\rho}$ and $\beta<\alpha$ (by Karamata's theorem~\cite[Thm~1.5.11]{MR1015093}, we can take $\beta=\alpha$ if the normalising function $g$ of $X$ satisfies $\liminf_{t\da0}t^{-1/\alpha}g(t)>0$). Moreover, the assumptions of cases (b-i) and (b-ii) hold for processes in the domain of normal attraction (i.e. if the normalising function equals $g(t)= t^{1/\alpha}$ for all $t>0$) with $\rho\in(0,1)$ and $\beta=\alpha\in(0,1]$, see~\cite[Thm~2]{MR3784492}. In particular, these assumptions are satisfied by stable processes with $\alpha\in(0,1]$ and $\rho\in(0,1)$.
\item[{{(b)}}] Both integrals in part (a) of Corollary~\ref{cor:power_func_liminf_post_min} are finite or infinite simultaneously whenever $f'$ is regularly varying at $0$ with nonzero index by Karamata's theorem~\cite[Thm~1.5.11]{MR1015093}. Thus, in that case, under the conditions of either (b-i) or (b-ii), the limit $\liminf_{t\da 0}(C'_{t+\tau_s}-s)/f(t)$ equals $0$ or $\infty$ according to whether $\int_0^1t^{-1/\beta}f(t)\D t$ is infinite or finite, respectively.
\item[{{(c)}}] The case $\beta>1$ is not considered in Corollary~\ref{cor:power_func_liminf_post_min}(a) and (b-ii) since in this case we would have $\mL^+(\mS)=\emptyset$ by~\cite[Prop.~1.6]{SmoothCM}. \qedhere
\end{itemize}
\end{remark}

\begin{proof}[Proof of Corollary~\ref{cor:power_func_liminf_post_min}]
Assume without loss of generality that $s=0\in\mL^+(\mS)$ (equivalently, we consider the process $(X_t-st)_{t\ge 0}$ for $s\in\mL^+(\mS)$).

(a) Our assumptions and~\cite[Thm~3.1]{picard_1997} show that the density $x\mapsto p_X(t,x)$ of $X_t$ exists for $t>0$ and moreover $\sup_{x\in\R}p_X(t,x)\le Ct^{-1/\beta}$ for some $C>0$ and all $t\in(0,1]$. Thus,~\eqref{eq:suff_low_post_min} is implied by 
\begin{equation}
\label{eq:simple_suff_cond_post_min_density}
\int_0^1 \int_{tf(t/2)}^t\frac{1}{f^{-1}(x/t)}t^{-1/\beta} \D x \D t
=\int_0^1 \int_{t/2}^1\frac{f'(y)}{y}t^{1-1/\beta}\D y \D t<\infty,
\end{equation} 
where we have used the change of variable $x=tf(y)$. Similarly, the bound on the density of $X_t$ shows that condition~\eqref{eq:post-min-Pi-large} holds if $\int_0^1 t^{-1/\beta}f(t)\D t<\infty$. Thus, the result follows from Theorem~\ref{thm:post-min-lower}.

(b) In either case (i) or (ii), our assumptions and~\cite[Thm~4.3]{picard_1997} show that $Ct^{-1/\beta}\le p_X(t,x)$ 
for some $C>0$ and all $|x|\le t^{1/\beta}$. Thus $\p(0<X_t\le tf(t/c))\ge ((tf(t/c))\wedge t^{1/\beta})Ct^{-1/\beta}$, implying that~\eqref{eq:post-min-Pi-large} fails for some $c>0$ whenever $\int_0^1 ((t^{-1/\beta}f(t/c))\wedge t^{-1})\D t=\infty$. A simple change of variables shows that this integral is either finite for all $c>0$ or infinite for all $c>0$. The result then follows from Theorem~\ref{thm:post-min-lower}(ii).
\end{proof}

The following is another simple corollary of Theorem~\ref{thm:post-min-lower}. This result can also be established using similar arguments to those used in~\cite[Cor.~3]{MR1747095}, see the discussion ensuing the proof of~\cite[Cor.~3]{MR1747095}.

\begin{corollary}
\label{cor:stable_liminf_post_min}
Let $X$ be a Cauchy process, $f$ be as in Theorem~\ref{thm:post-min-lower} and pick $s\in\R$. Then the limit $\liminf_{t\da 0}(C'_{t+\tau_s}-s)/f(t)$ equals $0$ (resp. $\infty$) a.s. if $\int_0^1 t^{-1}f(t)\D t$ is infinite (resp. finite).
\end{corollary}

\begin{proof}
Assume without loss of generality that $s=0$. Then the law of $X_t/t$ does not depend on $t>0$ and hence the integral in~\eqref{eq:suff_low_post_min} equals
\[
\int_0^1\E\bigg[\frac{\1_{\{t/2<f^{-1}(X_1)\le 1\}}}{f^{-1}(X_1)}\bigg]\D t
=\E\bigg[\int_0^1\frac{\1_{\{t/2<f^{-1}(X_1)\le 1\}}}{f^{-1}(X_1)}\D t\bigg]
\le 2\p(X_1\in(0,1])<\infty.
\]
Moreover, condition~\eqref{eq:post-min-Pi-large} simplifies to $\int_0^1\p(0<X_1\le f(t/c))t^{-1}\D t<\infty$, which is equivalent to the integral $\int_0^1t^{-1}f(t/c)\D t$ being finite since $X_1$ has a bounded density that is bounded away from zero on $[0,1]$. The change of variables $t'=t/c$ shows that this integral is either finite for all $c>0$ or infinite for all $c>0$. Thus, Theorem~\ref{thm:post-min-lower} gives the result.
\end{proof}

\subsection{Regime (FS): upper functions at time \texorpdfstring{$\tau_s$}{tau}}
\label{subsec:upper-post-min}

The upper fluctuations of $C'_{t+\tau_s}-s$ are harder to describe than the lower fluctuations studied in Subsection~\ref{subsec:lower-post-min} above. The main reason for this is that in Theorem~\ref{thm:upper_fun_C'_post_min} below the $\limsup$ of $C'$ at a vertex time $\tau_s$ can be expressed in terms of the $\liminf$ of the vertex time process $\tau$, which requires strong two-sided control on the Laplace exponent $\Phi_{u+s}(w)-\Phi_{s}(w)$, defined in~\eqref{eq:cf_tau0}, of the variable $\tau_{u+s}-\tau_s$ 
as $w\to\infty$ and $u\da 0$. (In the proof of Theorem~\ref{thm:post-min-lower}, $\limsup$ of the vertex time process $\tau$ is needed, which is easier to control.) In turn, by~\eqref{eq:cf_tau0}, this requires sharp two-sided estimates on the probability $\p(0<X_t-st\le ut)$ as a function of $(u,t)$ for small $u,t>0$. In particular, it is important to have strong control on the density of $X_t$ for small $t>0$ on the ``pizza slice'' $\{(t,x):s<x/t\le u+s\}$ as $u\da 0$. We establish these estimates for the processes in the domain of attraction of an $\alpha$-stable process, leading to Theorem~\ref{thm:upper_fun_C'_post_min} below. 

We denote by $\mZ_{\alpha,\rho}$ the class of L\'evy processes in the small-time domain of attraction of an $\alpha$-stable process with positivity parameter $\rho\in[0,1]$ (see~\cite[Eq.~(8)]{MR3784492}). In the case $\alpha<1$, relevant in the regime (FS) at slope $s$ equal to the natural drift $\gamma_0$, for each L\'evy process $X\in\mZ_{\alpha,\rho}$ there exists a normalising function $g$ that is regularly varying at $0$ with index $1/\alpha$ and an $\alpha$-stable process $(Z_u)_{u\in[0,T]}$ with $\rho=\p(Z_1>0)\in[0,1]$ such that the weak convergence $((X_{ut}-\gamma_0ut)/g(t))_{u\in[0,T]}\cid(Z_u)_{u\in[0,T]}$ holds as $t\da 0$. Given $X\in\mZ_{\alpha,\rho}$ with normalising function $g$, we define $G(t)\coloneqq t/g(t)$ for $t\in(0,\infty)$. 

\begin{theorem}\label{thm:upper_fun_C'_post_min}
Suppose $X\in\mZ_{\alpha,\rho}$ for some $\alpha\in(0,1)$ and $\rho\in(0,1]$. Define $f: (0,1) \to (0,\infty)$ through $f(t)\coloneqq 1/G(t\log^p (1/t))$, $t\in(0,1)$, for some $p\in\R$. Then the following hold for $s=\gamma_0$:
\begin{itemize}[leftmargin=2.5em, nosep]
\item[\nf(i)] if $p>1/\rho$, then $\limsup_{t\da 0}(C'_{t+\tau_s}-s)/f(t)=0$ a.s.,
\item[\nf(ii)] if $p<1/\rho$, then $\limsup_{t\da 0}(C'_{t+\tau_s}-s)/f(t)=\infty$ a.s.
\end{itemize}
\end{theorem}

The class $\mZ_{\alpha,\rho}$ is quite large and the assumption $X\in\mZ_{\alpha,\rho}$ is essentially reduced to the L\'evy measure of $X$ being regularly varying at $0$, see~\cite[\S4]{MR3784492} for a full characterisation of this class. In particular, $\alpha$ agrees with the Blumenthal--Getoor index $\beta_\BG$ defined in~\eqref{eq:beta} below. Moreover, for $\alpha<1$ and $\rho \in (0,1]$, the assumption $X\in\mZ_{\alpha,\rho}$ implies that $X$ is of finite variation with $\p(X_t-\gamma_0t>0)\to\rho$ as $t\da 0$, implying  $\mL^+(\mS)=\{\gamma_0\}$ by~\cite[Prop.~1.3 \& Cor.~1.4]{SmoothCM}. 

Note that the function $f$ in Theorem~\ref{thm:upper_fun_C'_post_min} is regularly varying at $0$ with index $1/\alpha-1$. The appearance of the positivity parameter $\rho$, a nontrivial function of the L\'evy measure of $X$, in Theorem~\ref{thm:upper_fun_C'_post_min} suggests that the upper fluctuations of $C'$ at time $\tau_s$ (for $s=\gamma_0$) are more delicate than its lower fluctuations described in Theorem~\ref{thm:lower_fun_C'}. Indeed, if $X\in\mZ_{\alpha,\rho}$ is in the domain of normal attraction (i.e. $g(t)=t^{1/\alpha}$) and $\rho\in(0,1)$, then the fluctuations of $C'$ at vertex time $\tau_s$, characterised by Corollary~\ref{cor:power_func_liminf_post_min}(a) \& (b-ii) (with $\beta=\alpha$) and Remark~\ref{rem:power_func_liminf_post_min}(a), do not involve parameter $\rho$. In particular, by Theorem~\ref{thm:upper_fun_C'_post_min} and Corollary~\ref{cor:power_func_liminf_post_min}(b-ii), we have $\liminf_{t\da 0}(C'_{t+\tau_s}-s)/f(t)=0$ and $\limsup_{t\da 0}(C'_{t+\tau_s}-s)/f(t)=\infty$ a.s. for $f(t)=t^{1/\alpha-1}\log^{q}(1/t)$ and any $q\in[-1,(1/\alpha-1)/\rho)$, demonstrating the gap between the lower and upper fluctuations of $C'$ at vertex time $\tau_s$. 

\begin{remark}
\label{rem:exclusions-tau}\phantom{empty}
\begin{itemize}[leftmargin=2em, nosep]
\item[(a)] The case where $X$ is attracted to Cauchy process with $\alpha=1$ is expected to hold for the functions $f$ in Theorem~\ref{thm:upper_fun_C'_post_min}. For such $X\in\mZ_{1,\rho}$, a multitude of cases arise including $X$ having (i) less activity (e.g., $X$ is of finite variation), (ii) similar amount of activity (i.e., $X$ is in the domain of normal attraction) or (iii) more activity than Cauchy process (see, e.g.~\cite[Ex.~2.1--2.2]{SmoothCM}). In terms of the normalising function $g$ of $X$, these cases correspond to the limit $\lim_{t\da 0}t^{-1/\alpha}g(t)$ being equal to: (i) zero, (ii) a finite and positive constant or (iii) infinity. (Recall that in cases (ii) and (iii) $X$ is strongly eroded with $\mL^+(\mS)=\R$, see~\cite[Ex.~2.1--2.2]{SmoothCM}, and in case (i) $X$ may be strongly eroded, by~\cite[Thm~1.8]{SmoothCM}, or of finite variation with $\mL^+(\mS)=\{\gamma_0\}$ by~\cite[Prop~138]{SmoothCM} and the fact that $\lim_{t\da0}\p(X_t>0)=\rho\in(0,1)$.) However, we stress that our methodology can be used to obtain a description of the lower fluctuations of $C'$ at $\tau_s$ in cases (i), (ii) and (iii). This would require an application of Theorem~\ref{thm:limsup_L} along with two-sided estimates of the Laplace exponent $\Phi$ of the vertex time process in~\eqref{eq:cf_tau0}, generalising  Lemma~\ref{lem:asymp_equiv_Psi_domstable_post_min} to the case $\alpha=1$. In the interest of brevity we do not give the details of this extension.
\item[(b)] The boundary case $p=1/\rho$ can be analysed along similar lines. In fact, our methods can be used to get increasingly sharper results, determining the value of $\limsup_{t\da 0}(C'_{t+\tau_s}-s)/f(t)$ for functions $f$ containing powers of iterated logarithms, when stronger control over the densities of the marginals of $X$ is available. Such refinements are possible when $X$ is a stable process cf. Section~\ref{sec:concluding_rem}. In particular, we may prove the following law of iterated logarithm given in~\cite[p.~54]{MR1747095} for a Cauchy process $X$ with density $x\mapsto p_X(t,x)$ at time $t>0$: for any $s\in\R$ and the function $f(t)=(\log\log\log(1/t))/\log(1/t)$, 
we have $\limsup_{t\da 0}(C'_{t+\tau_s}-s)/f(t)=1/p_X(1,s)$ a.s. \qedhere
\end{itemize}
\end{remark}



\subsection{Regime (IS): upper functions at time {\nf{0}}}
\label{subsec:upper_fun_C}

Throughout this subsection we assume $X$ has infinite variation, equivalent to $\liminf_{t\da0}C'_t=-\infty$ a.s.~\cite[Sec.~1.1.2]{SmoothCM}. The following theorem describes the upper fluctuations of $C'_t$ as $t\da 0$. 

\begin{theorem}
\label{thm:C'_limsup}
Let $f$ be continuous and increasing with  $f(0)=0=\lim_{c\da0}\limsup_{t\da0}f(ct)/f(t)$ and $f(t)\le 1=f(1)$ for $t\in(0,1]$. Let $c>0$, denote $F(t)\coloneqq t/f(t)$ for $t>0$ and consider the conditions
\begin{gather}
\label{eq:C'_large}
\int_0^1 \p(X_t\le -cF(t))\frac{\D t}{t}<\infty,\\
\label{eq:C'_var}
\int_0^1 \E[(X_t/F(t))^2\1_{\{-2F(t)<X_t\le -t\}}]\frac{\D t}{t}<\infty,\\
\label{eq:C'_mean}
2^n\int_0^{2^{-n}}
    \p(-t/f(2^{-n})\ge X_t>-2F(t/2))\D t
\to 0, \quad \text{ as }n \to \infty.
\end{gather}
Then the following statements hold.
\begin{itemize}[leftmargin=2.5em, nosep]
\item[{\nf(i)}] If~\eqref{eq:C'_large}--\eqref{eq:C'_mean} hold for $c=1$ and $f$ is concave, then $\limsup_{t \da 0} |C_t'|f(t)=0$ a.s.
\item[{\nf(ii)}] If~\eqref{eq:C'_large} fails for all $c>0$, then $\limsup_{t\da0} |C_t'|f(t)=\infty$ a.s.
\item[{\nf(iii)}] If $\limsup_{t \da 0} |C_t'|f(t)<1$ a.s., then~\eqref{eq:C'_large} holds for any $c>1$.
\end{itemize}
\end{theorem}

Some remarks are in order.

\begin{remark}
\label{rem:limsup_cond_3}\phantom{empty}
\begin{itemize}[leftmargin=2em, nosep]
\item[{(a)}] Any continuous regularly varying function $f$ of index $r>0$ satisfies the assumption in the theorem, 
see Remark~\ref{rem:limsup_cond_post_min}(a) above.
\item[{(b)}]
The proof of Theorem~\ref{thm:C'_limsup} is based on the analysis of the upper fluctuations of the vertex time $\tau_{-1/u}$ as $u\da0$. The interpretation and purpose of conditions~\eqref{eq:C'_large}--\eqref{eq:C'_mean} are analogous to those of conditions~\eqref{eq:post-min-Pi-large}--\eqref{eq:post-min-Pi-mean}, respectively, see Remark~\ref{rem:limsup_cond_post_min}(b) above.
\item[{(c)}] Note that~\eqref{eq:C'_mean} holds if $\int_0^1\p(-2F(2^{-n}t/2)< X_{2^{-n}t}\le -tF(2^{-n}))\D t\to 0$ as $n\to\infty$, which, by the dominated convergence theorem, is the case if $\p(-2F(u/2)<X_{u}\le -tF(u/t))\to 0$ as $u\da 0$ for a.e. $t\in(0,1)$.
\item[{(d)}] The assumed concavity of $f$ in part (ii) can be dropped by modifying assumption~\eqref{eq:C'_var} into a condition involving the inverse of $f$ (cf. Corollary~\ref{cor:L_liminf} and Proposition~\ref{prop:Y_limsup}). We do not make this explicit in the statement of Theorem~\ref{thm:C'_limsup} because the functions of interest in this context are typically concave.\qedhere
\end{itemize}
\end{remark}

\subsubsection{Simple sufficient conditions for the assumptions of Theorem~\ref{thm:C'_limsup}}
\label{subsec:simple_suff_cond_time_0}

The tail probabilities of $X_t$ appearing in the assumptions of Theorem~\ref{thm:C'_limsup} are not analytically available in general. In this subsection we present sufficient conditions, in terms of the generating triplet $(\gamma,\sigma^2,\nu)$ of $X$, implying the assumptions in~\eqref{eq:C'_large}--\eqref{eq:C'_mean} of Theorem~\ref{thm:C'_limsup}. Recall that $\ov\sigma^2(\ve) 
=\sigma^2 + \int_{(-\ve,\ve)}x^2\nu(\D x)$ for $\ve>0$, and define:
\begin{equation}
\label{eq:ov_functions}
\ov \gamma(\ve)
\coloneqq \gamma-\int_{(-1,1)\setminus(-\ve,\ve)}x\nu(\D x),
\qquad
\ov\nu(\ve) 
\coloneqq \nu(\R\setminus(-\ve,\ve)),
\quad\text{for all }\ve>0.
\end{equation}
Let $f$ and $F$ be as in Theorem~\ref{thm:C'_limsup} and note that $F(t)\in(0,1]$ since $f$ is concave with $f(1)=1$. The inequalities in Lemma~\ref{lem:upper_tail_bound} (with $p=2$, $\ve=F(t)\in(0,1]$ and $K=cF(t)$), applied to $\p(|X_t|\ge cF(t))$ and $\E[\min\{X_t^2,4F(t)^2\}]\ge\E[X_t^2\1_{\{|X_t|\le 2F(t)\}}]$, show that the condition 
\begin{equation}
\label{eq:C'_suff_var}
\int_0^1 \big[F(t)^{-2}\big(\ov \gamma^2(F(t))t+\ov\sigma^2(F(t))\big) + \ov\nu(F(t))\big]\D t
<\infty,
\end{equation}
implies~\eqref{eq:C'_large}--\eqref{eq:C'_var}. Similarly, by Remark~\ref{rem:limsup_cond_3}(c) and Lemma~\ref{lem:upper_tail_bound}, the following condition implies~\eqref{eq:C'_mean}:
\begin{equation}
\label{eq:C'_suff_mean}
\big[F(t)^{-2}\big(\ov \gamma^2(F(t))t+\ov\sigma^2(F(t))\big) 
    + \ov\nu(F(t))\big]t
\to 0,
\qquad\text{as }t\da 0.
\end{equation}
These simplifications lead to the following corollary.

\begin{corollary}
\label{cor:power_func_limsup}
Suppose $\ov\nu(\ve)+\ve^{-2}\ov\sigma^2(\ve)+\ve^{-1}|\ov\gamma(\ve)|=\Oh(\ve^{-\beta})$ as $\ve\da 0$ for some $\beta\in[1,2]$ and, as before, let $F(t)=t/f(t)$. If we have $F(t)^{-\beta}t\to 0$ as $t\to 0$ and $\int_0^1F(t)^{-\beta}\D t<\infty$, then $\limsup_{t\da 0}|C'_t|f(t)=0$ a.s.
\end{corollary}

\begin{proof}
By virtue of Theorem~\ref{thm:C'_limsup}(i), it suffices to verify~\eqref{eq:C'_suff_var} and~\eqref{eq:C'_suff_mean}. By assumption, we have $[F(t)^{-2}\ov\sigma^2(F(t))+\ov\nu(F(t))]t=\Oh(F(t)^{-\beta}t)$  and $F(t)^{-2}\ov\gamma(F(t))^2t^2=\Oh((F(t)^{-\beta}t)^2)$, which tend to $0$ as $t\da 0$, implying~\eqref{eq:C'_suff_mean}. Condition~\eqref{eq:C'_suff_var} follows similarly, completing the proof.
\end{proof}


Define the Blumenthal--Getoor index $\beta_\BG\in[0,2]$ of $X$~\cite{MR0123362} as follows:
\begin{equation}
\label{eq:beta}
\beta_\BG \coloneqq  \inf\{q\in[0,2]\,:\,I_q<\infty\},
\qquad\text{where}\qquad
I_q\coloneqq \int_{(-1,1)\setminus\{0\}}|x|^q\nu(\D x),
\quad
q>0.
\end{equation}
Note that, in our setting, $X$ has infinite variation and hence $\beta_\BG\ge 1$. Since $I_\beta<\infty$ for any $\beta>\beta_\BG$, \cite[Lem.~1]{LevySupSim} shows that $\beta$ satisfies the assumptions of Corollary~\ref{cor:power_func_limsup}. Hence $\limsup_{t\da0}|C'_t|t^{p}=0$ a.s. for any $p>1-1/\beta_\BG\in[0,1/2]$ by Corollary~\ref{cor:power_func_limsup}.

Stronger results are possible when stronger conditions are imposed on the law of $X$. For instance, for stable processes we have the following consequence of Theorem~\ref{thm:C'_limsup}.

\begin{corollary}
\label{cor:stable_limsup}
Let $X$ be an $\alpha$-stable process with $\alpha\in[1,2)$. Then the following statements hold.
\begin{itemize}[leftmargin=2em, nosep]
\item[\nf{(a)}] If $t\mapsto t^{-1/\alpha}F(t)$ is bounded as $t\da 0$, then $\limsup_{t\da 0}|C'_t|f(t)=\infty$ a.s. 
\item[\nf{(b)}] If $t^{-1/\alpha}F(t)\to\infty$ as $t\da0$ and $X$ is not spectrally positive, then the limit $\limsup_{t\da 0}|C'_t|f(t)$ is equal to $\infty$ (resp. $0$) a.s. if the integral $\int_0^1F(t)^{-\alpha}\D t$ is infinite (resp. finite). 
\end{itemize}
\end{corollary}

\begin{proof}
The scaling property of $X$ gives $\p(X_t\le -cF(t))=\p(X_1\le -ct^{-1/\alpha}F(t))$ for any $c,t>0$. If $t\mapsto t^{-1/\alpha}F(t)$ is bounded, then $\liminf_{t\da 0}\p(X_t\le -cF(t))>0$ making~\eqref{eq:C'_large} fail for all $c>0$. In that case, we have $\limsup_{t\da 0}|C'_t|f(t)=\infty$ a.s. by Theorem~\ref{thm:C'_limsup}(ii), proving part (a). 

To prove part (b), suppose $X$ is not spectrally positive and let $t^{-1/\alpha}F(t)\to\infty$ as $t\da 0$. Then $x^\alpha\p(X_1\le -x)$ converges to a positive constant as $x\to\infty$, implying the following equivalence: $\int_0^1t^{-1}\p(X_t\le -ct^{-1/\alpha}F(t))\D t<\infty$ if and only if $\int_0^1 F(t)^{-\alpha}\D t<\infty$, where we note that the last integral does not depend of $c>0$. If $\int_0^1 F(t)^{-\alpha}\D t<\infty$, then~\eqref{eq:C'_suff_var}--\eqref{eq:C'_suff_mean} hold and Theorem~\ref{thm:C'_limsup}(i) gives $\limsup_{t\da 0}|C'_t|f(t)=0$ a.s. If instead $\int_0^1 F(t)^{-\alpha}\D t=\infty$, then $\int_0^1t^{-1}\p(X_t\le -ct^{-1/\alpha}F(t))\D t=\infty$ for all $c>0$, so Theorem~\ref{thm:C'_limsup}(ii) implies that $\limsup_{t\da 0}|C'_t|f(t)=\infty$ a.s., completing the proof.
\end{proof}

For Cauchy process (i.e. $\alpha=1$), Corollary~\ref{cor:stable_limsup} contains the dichotomy in~\cite[Cor.~3]{MR1747095} for the upper functions of $C'$ at time $0$. We note here that results analogous to Corollary~\ref{cor:stable_limsup} can be derived for a spectrally positive stable process $X$ (and for Brownian motion), using the exponential (instead of polynomial) decay of the probability $\p(X_1\le x)$ in $x$ as $x\to-\infty$, see~\cite[Thm~4.7.1]{MR1745764}.

\subsection{Regime (IS): lower functions at time {\nf{0}}}
\label{subsec:lower_fun_C'}

As explained before, obtaining fine conditions for the lower fluctuations of $C'$ is more delicate than in the case of upper fluctuations of $C'$ at $0$. The main reason is that the proof of Theorem~\ref{thm:lower_fun_C'} requires strong control on the Laplace exponent $\Phi_u(w)$ of $\tau_u$, defined in~\eqref{eq:cf_tau0}, as $w\to\infty$ and $u\to-\infty$. This in turn requires sharp two-sided estimates on the negative tail probability $\p(X_t\le ut)$ as a function of $(u,t)$ as $u\to-\infty$ and $t\da 0$ jointly. 

Due to the necessity of such strong control, in the following result we assume $X\in\mZ_{\alpha,\rho}$ for some $\alpha>1$. In other words, we assume there exist some normalising function $g$ that is regularly varying at $0$ with index $1/\alpha$ and an $\alpha$-stable process $(Z_s)_{s \in [0,T]}$ with $\rho=\p(Z_1>0)\in(0,1)$ such that $(X_{ut}/g(t))_{u \in [0,T]} \cid (Z_u)_{u \in [0,T]}$ as $t\da 0$. Recall that $G(t)=t/g(t)$ for $t>0$.

\begin{theorem}\label{thm:lower_fun_C'}
Let $X\in\mZ_{\alpha,\rho}$ for some $\alpha\in(1,2]$ (and hence $\rho\in(0,1)$). Let $f: (0,1) \to (0,\infty)$ be given by $f(t)\coloneqq G(t\log^p(1/t))$, for some $p\in\R$ and all $t\in(0,1)$. Then the following statements hold:
\begin{itemize}[leftmargin=2.5em, nosep]
\item[\nf(i)] if $p>1/(1-\rho)$, then $\liminf_{t\da 0}|C'_t|f(t)=\infty$ a.s.,
\item[\nf(ii)] if $p<1/(1-\rho)$, then $\liminf_{t\da 0}|C'_t|f(t)=0$ a.s.
\end{itemize}
\end{theorem}

\begin{remark}
\label{rem:exclusions-0}\phantom{empty}
\begin{itemize}[leftmargin=2em, nosep]
\item[(a)] The assumption $X\in\mZ_{\alpha,\rho}$ for some $\alpha>1$ implies that $X$ is of infinite variation. Note that the function $f$ in Theorem~\ref{thm:lower_fun_C'} is regularly varying at $0$ with index $1-1/\alpha$. The `negativity' parameter $1-\rho=\lim_{t\da 0}\p(X_t<0)\in(0,1)$ is a nontrivial function of the L\'evy measure of $X$. The fact that $1-\rho$ features as a boundary point in the power of the logarithmic term in Theorem~\ref{thm:lower_fun_C'} indicates that the lower fluctuations of $C'$ at time $0$ depends in a subtle way on the characteristics of $X$. Such dependence is, for instance, not present for the upper fluctuations of $C'$ at time $0$ when $X$ is $\alpha$-stable, see Corollary~\ref{cor:stable_limsup} above. Indeed, for an $\alpha$-stable process $X$, Theorem~\ref{thm:lower_fun_C'} and Corollary~\ref{cor:stable_limsup}(b) show that $\liminf_{t\da 0}|C'_{t}|f(t)=0$ and $\limsup_{t\da 0}|C'_{t}|f(t)=\infty$ a.s. for $f(t)=t^{1-1/\alpha}\log^{q}(1/t)$ and any $q\in[-1/\alpha,(1-1/\alpha)/(1-\rho))$, demonstrating the gap between the lower and upper fluctuations of $C'$ at time $0$.
\item[(b)] The case where $X$ is attracted to Cauchy process with $\alpha=1$ is expected to hold for the functions $f$ in Theorem~\ref{thm:lower_fun_C'}. As explained in Remark~\ref{rem:exclusions-tau}(a) above, many cases arise, with even some abrupt processes being attracted to Cauchy process (see~\cite[Ex.~2.2]{SmoothCM}). We again stress that, in this case, our methodology can be used to obtain a description of the upper fluctuations of $C'$ at time~$0$ via Theorem~\ref{thm:Y_limsup} and two-sided estimates, analogous to Lemma~\ref{lem:asymp_equiv_Phi_domstable}, of the Laplace exponent $\Phi$ in~\eqref{eq:cf_tau0} of the vertex time process. In the interest of brevity, we omit the details of such extensions.
\item[(c)] As with Theorem~\ref{thm:upper_fun_C'_post_min} above (see Remark~\ref{rem:exclusions-tau}(b)), the boundary case $p=1/(1-\rho)$ in Theorem~\ref{thm:lower_fun_C'} can be analysed along similar lines. In fact, our methods can be used to get increasingly sharper results for the lower fluctuations of $C'$ at time $0$ when stronger control over the negative tail probabilities for the marginals $X$ is available. Such improvements are possible, for instance, when $X$ is $\alpha$-stable. We decided to leave such results for future work in the interest of brevity. For completeness, however, we mention that the following law of iterated logarithm proved in~\cite[Cor.~3]{MR1747095} can also be proved using our methods (see Example~\ref{ex:Cauchy} below): $\liminf_{t\da 0}|C'_t|f(t)=p_X(1,0)$ a.s., where $x\mapsto p_X(t,x)$ is the density of $X_t$.\qedhere
\end{itemize}
\end{remark}

\subsection{Upper and lower function of the L\'evy path at vertex times}
\label{subsec:applications} 

In this section we establish consequences for the lower (resp. upper) fluctuations of the L\'evy path at vertex time $\tau_s$ (resp. time $0$) in terms of those of $C'$. Recall $X_{t-}\coloneqq \lim_{u \ua t} X_u$ for $t > 0$ (and $X_{0-}\coloneqq X_0$) and define $m_s\coloneqq \min\{X_{\tau_s},X_{\tau_s-}\}$ for $s\in\mL^+(\mS)$. 

\begin{lemma}\label{lem:upper_fun_Lev_path_post_slope}
Suppose $s \in \mL^+(\mS)$. Let the function $f:[0,\infty) \to [0,\infty)$ be continuous and increasing and define the function $\tilde f(t)\coloneqq\int_0^t f(u)\D u$, $t\ge 0$. Then the following statements hold for any $M>0$.
\begin{itemize}[leftmargin=2.5em, nosep]
\item[\nf(i)] If $\liminf_{t\da 0} (C'_{t+\tau_s}-s)/f(t)>M$ a
.s. then $\liminf_{t \da 0} (X_{t+\tau_s}-m_s-st)/\tilde f(t)\ge M$ a.s.
\item[\nf(ii)] If $\limsup_{t\da 0} (C'_{t+\tau_s}-s)/f(t)<M$ a.s. then $\liminf_{t \da 0} (X_{t+\tau_s}-m_s-st)/\tilde f(t)\le M$ a.s.
\end{itemize}
\end{lemma}

The proof of Lemma~\ref{lem:upper_fun_Lev_path_post_slope} is pathwise. The lemma yields the following implications
\begin{itemize}[leftmargin=2.5em, nosep]
\item[\nf(i)] $\liminf_{t\da 0} (C'_{t+\tau_s}-s)/f(t)=\infty\implies\liminf_{t \da 0} (X_{t+\tau_s}-m_s-st)/\tilde f(t)=\infty$,
\item[\nf(ii)] $\limsup_{t\da 0} (C'_{t+\tau_s}-s)/f(t)=0\implies\liminf_{t \da 0} (X_{t+\tau_s}-m_s-st)/\tilde f(t)=0$.
\end{itemize}
The upper fluctuations of $X$ at vertex time $\tau_s$ cannot be controlled via the fluctuations of $C'$ since the process may have large excursions away from its convex minorant between contact points. Moreover, the limits $\liminf_{t\da 0} (C'_{t+\tau_s}-s)/f(t)=0$ or $\limsup_{t\da 0} (C'_{t+\tau_s}-s)/f(t)=\infty$, do not provide sufficient information to ascertain the value of the lower limit $\liminf_{t\da 0}(X_{t+\tau_s}-m_s-st)/\tilde f(t)$, since this limit may not be attained along the contact points between the path and its convex minorant. 

Theorems~\ref{thm:post-min-lower} a
nd~\ref{thm:upper_fun_C'_post_min} give sufficient conditions, in terms of the law of $X$, for the assumptions in Lemma~\ref{lem:upper_fun_Lev_path_post_slope} to hold. This leads to the following corollaries.

\begin{corollary}
\label{cor:post_tau_s_Levy_path}
Let $s \in \mL^+(\mS)$ and let $f$ be a continuous and increasing function with $f(0)=0=\lim_{c\da0}\limsup_{t\da0}f(ct)/f(t)$, $f(1)=1$ and $f(t)\le 1$ for $t\in(0,1]$. If conditions~\eqref{eq:post-min-Pi-large}--\eqref{eq:post-min-Pi-mean} hold for $c=1$, then $\liminf_{t \da 0} (X_{t+\tau_s}-m_s-st)/\tilde f(t)=\infty$ a.s. where we denote $\wt f(t)\coloneqq \int_0^t f(u) \D u$.
\end{corollary}

Denote by $\varpi(t)\coloneqq t^{-1/\alpha}g(t)$ the slowly varying (at $0$) component of the normalising function $g$ of a process in the class $\mZ_{\alpha,\rho}$. Recall that $G(t)=t/g(t)$ for $t>0$.

\begin{corollary} 
\label{cor:post-tau_s-Levy-path-attraction}
Let $X\in\mZ_{\alpha,\rho}$ for some $\alpha\in(0,1)$ and $\rho\in(0,1]$. Given $p\in\R$, denote $\tilde f(t)\coloneqq\int_0^t G(u\log^p(u^{-1}))^{-1}\D u$ for $t> 0$. Then the following statements hold for $s=\gamma_0$.
\begin{itemize}[leftmargin=2.5em, nosep]
    \item[\nf(i)] If $p>1/\rho$, then $\liminf_{t \da 0} (X_{t+\tau_{s}}-m_{s}-st)/\tilde f(t)=0$ a.s.
    \item[\nf(ii)] If $\alpha\in (1/2,1)$, $p<-\alpha/(1-\alpha)$ and $(\varpi(c/t)/\varpi(1/t)-1)\log\log(1/t) \to 0$ as $t \da 0$ for some $c\in(0,1)$, then $\liminf_{t \da 0} (X_{t+\tau_{s}}-m_{s}-st)/\tilde f(t)=\infty$ a.s.
    \item[\nf{(iii)}] If $\alpha \in (0,1/2]$, then $\liminf_{t \da 0} (X_{t+\tau_{s}}-m_{s}-st)/ t^q=\infty$ a.s. for any $q>1/\alpha\ge 2$.
\end{itemize}
\end{corollary}

\begin{remark}
\phantom{empty}
\begin{itemize}[leftmargin=2em, nosep]
\item[(a)] The function $\tilde f$ is regularly varying at $0$ with index $1/\alpha$. This makes conditions in Corollary~\ref{cor:post-tau_s-Levy-path-attraction} nearly optimal in the following sense: the polynomial rate in all three cases is either $1/\alpha$ (cases (i) and (ii) in Corollary~\ref{cor:post-tau_s-Levy-path-attraction}) or arbitrarily close to it (case (iii) in Corollary~\ref{cor:post-tau_s-Levy-path-attraction}). If $\alpha>1/2$, then the gap is in the power of the logarithm in the definition of $\tilde f$.
\item[(b)] When the natural drift $\gamma_0=0$, Corollary~\ref{cor:post-tau_s-Levy-path-attraction} describes the lower fluctuations (at time $0$) of the post-minimum process $X^\ra=(X^\ra_t)_{t\in[0,T-\tau_0]}$ given by $
X^\ra_t\coloneqq X_{t+\tau_0}-m_0$ (note that $m_0=\inf_{t\in[0,T]}X_t$). The closest result in this vein is~\cite[Prop.~3.6]{MR1947963} where Vigon shows that, for any infinite variation L\'evy process $X$ and $r>0$, we have $\liminf_{t\da 0}X^\ra_t/t\ge r$ a.s. if and only if $\int_0^1\p(X_t/t\in[0,r])t^{-1}\D t<\infty$. Our result considers non-linear functions and a large class of finite variation processes.
\item[(c)] By~\cite[Thm~2]{MR3160578}, the assumption $X\in\mZ_{\alpha,\rho}$ and $\gamma_0=0$ 
implies that the post-minimum process, conditionally given $\tau_0$, is a L\'evy meander. Hence, Corollary~\ref{cor:post-tau_s-Levy-path-attraction} also describes the lower functions of the meanders of L\'evy processes in $\mZ_{\alpha,\rho}$. A similar remark applies to the results in Corollary~\ref{cor:post_tau_s_Levy_path}.\qedhere
\end{itemize}
\end{remark}

When $X$ has infinite variation, the process $X$ and $C$ touch each other infinitely often on any neighborhood of $0$ (see~\cite{SmoothCM}), leading to the following connection in small time between the paths of $X$ and its convex minorant $C$.

\begin{lemma}\label{lem:upper_fun_Lev_path}
Let the function $f:[0,\infty) \to [0,\infty)$ be continuous and increasing with $f(0)=0$ and finite $\tilde f(t)\coloneqq\int_0^t f(u)^{-1}\D u$, $t\ge 0$. Then the following statements hold for any $M>0$.

\begin{itemize}[leftmargin=2.5em, nosep]
\item[\nf(i)] If $\limsup_{t\da 0} |C_t'|f(t)<M$ a.s., then  $\limsup_{t \da 0} (-X_t)/\tilde f(t)\le M$ a.s.
\item[\nf(ii)] If $\liminf_{t\da 0} |C_t'|f(t)>M$ a.s., then $\limsup_{t \da 0} (-X_t)/\tilde f(t)\ge M$ a.s.
\end{itemize}
\end{lemma}

Theorem~\ref{thm:C'_limsup} and the corollaries thereafter give sufficient explicit conditions for the assumption in Lemma~\ref{lem:upper_fun_Lev_path}(i) to hold. Similarly, Theorem~\ref{thm:lower_fun_C'} gives a fine class of functions $f$ satisfying the assumption in Lemma~\ref{lem:upper_fun_Lev_path}(ii) for a large class of processes. Such conclusions on the fluctuations of the L\'evy path of $X$ would not be new as the fluctuations of $X$ at $0$ are already known, see~\cite{SoobinKimLeeLIL,MR2480786,MR2591911}. In particular, the upper functions of $X$ and $-X$ at time $0$ were completely characterised in~\cite{MR2591911} in terms of the generating triplet of $X$. Let us comment on some two-way implications of our results, the literature and Lemma~\ref{lem:upper_fun_Lev_path}.

\begin{remark}
\label{rem:upper_fun_Lev_path}\phantom{empty}
\begin{itemize}[leftmargin=2em, nosep]
\item[(a)] By~\cite{MR0002054}, the assumption in Theorem~\ref{thm:C'_limsup}(ii) implies that $\limsup_{t\da0}|X_t|/F(t)=\infty$ a.s. where we recall that $F(t)=t/f(t)$. Similarly, by~\cite{MR0002054}, if $\limsup_{t\da0}|X_t|/F(t)=\infty$ a.s. then the assumption in Theorem~\ref{thm:C'_limsup}(ii) must hold for either $X$ or $-X$, which, by time reversal, implies that at least one of the limits $\limsup_{t\da0} |C_t'|f(t)$ or $\limsup_{t\da0} |C_{T-t}'|f(t)$ is infinite a.s. This conclusion is similar to that of Lemma~\ref{lem:upper_fun_Lev_path}, the main difference being the use of either $\tilde f$ or $F$. Note however, that if $f$ is regularly varying with index different from $1$, then~\cite[Thm~1.5.11]{MR1015093} implies $\lim_{t\da 0}\tilde f(t)/F(t)\in(0,\infty)$.
\item[(b)] The contrapositive statements of Lemma~\ref{lem:upper_fun_Lev_path} give information on $C'$ in terms of $-X$. Indeed, if we have $\limsup_{t \da 0}(-X_t)/\tilde f(t)>0$, then $\limsup_{t\da 0} |C_t'|f(t)>0$. Similarly, if $\limsup_{t \da 0} (-X_t)/\tilde f(t)<\infty$, then $\liminf_{t\da 0} |C_t'|f(t)<\infty$.\qedhere
\end{itemize}
\end{remark}

The connections between the fluctuations of $X$ and those of $C'$ at time $0$ are intricate. Although the one-sided fluctuations of $X$ at $0$ were essentially characterised in~\cite[Thm~3.1]{MR2591911}, its combination with Lemma~\ref{lem:upper_fun_Lev_path} is not sufficiently strong to obtain conditions for any of the following statements: $\limsup_{t\da0}|C'_t|f(t)=0$, $\limsup_{t\da0}|C'_t|f(
t)>0$, $\liminf_{t\da0}|C'_t|f(t)<\infty$ or $\liminf_{t\da0}|C'_t|f(t)=\infty$ a.s. 

\section{Small-time fluctuations of non-decreasing additive processes}
\label{sec:additive}

Consider a pure-jump right-continuous non-decreasing additive (i.e. with independent and possibly non-stationary increments) process $Y=(Y_t)_{t\ge 0}$ with $Y_0=0$ a.s. and its mean jump measure $\Pi(\D t,\D x)$ for $(t,x)\in[0,\infty)\times(0,\infty)$, see~\cite[Thm~15.4]{MR1876169}. Then, by Campbell's formula~\cite[Lem.~12.2]{MR1876169}, its Laplace transform satisfies
\begin{equation}
\label{eq:defn_Psi_additive_process}
   \E\big[e^{-uY_t}\big]
=e^{-\Psi_{t}(u)},
\quad\text{where}\quad
\Psi_{t}(u)\coloneqq \int_{(0,\infty)}(1-e^{-ux})\Pi((0,t],\D x),
\quad \text{for any }u\ge 0. 
\end{equation}
Let $L_t\coloneqq \inf\{u>0: Y_u>t\}$ for $t\ge0$ (with convention $\inf\emptyset=\infty$) denote the right-continuous inverse of $Y$. Our main objective in this section is to describe the upper and lower fluctuations of $L$, extending known results for the case where $Y$ has stationary increments (making $Y$ a subordinator) in which case $\Pi(\D t,\D x)=\Pi((0,1],\D x)\D t$ for all $(t,x)\in[0,\infty)\times(0,\infty)$ (see e.g.~\cite[Thm 4.1]{MR1746300}). 

\subsection{Upper functions of {\nf{\emph{L}}}}
\label{subsec:upper_func_gen}

The following theorem is the main result of this subsection.

\begin{theorem}\label{thm:limsup_L}
Let $f:(0,1)\to(0,\infty)$ be increasing with $\lim_{t \downarrow 0}f(t)=0$ and $\phi:(0,\infty) \to (0,\infty)$ be decreasing with $\lim_{u \to \infty} \phi(u)=0$. Let the positive sequence $(\theta_n)_{n\in\N}$ satisfy $\lim_{n\to \infty}\theta_n=\infty$ and define the associated sequence $(t_n)_{n\in\N}$ given by $t_n\coloneqq \phi(\theta_n)$ for any $n \in \N$.\\
{\nf{(a)}} If $\sum_{n=1}^\infty \exp(\theta_n t_n-\Psi_{f(t_n)}(\theta_n))<\infty$ then $\limsup_{t \da 0}L_t/f(t)\leq \limsup_{n \to \infty}f(t_n)/f(t_{n+1})$ a.s.\\
{\nf{(b)}} If $\lim_{u\to\infty}\phi(u)u=\infty$, $\sum_{n=1}^\infty [\exp(-\Psi_{f(t_n)}(\theta_n))-\exp(-\theta_n t_n)]=\infty$ and $\sum_{n=1}^\infty \Psi_{f(t_{n+1})}( \theta_n)<\infty$, then $\limsup_{t \downarrow 0} L_t/f(t)\geq 1$ a.s.
\end{theorem}

\begin{remark}
\label{rem:limsup_L}\phantom{empty}
\begin{itemize}[leftmargin=2em, nosep]
\item[(a)] Theorem~\ref{thm:limsup_L} plays a key role in the proofs of Theorems~\ref{thm:upper_fun_C'_post_min} and~\ref{thm:lower_fun_C'}. Before applying Theorem~\ref{thm:limsup_L}, one needs to find appropriate choices of the free infinite-dimensional parameters $h$ and $(\theta_n)_{n\in\N}$. This makes the application of Theorem~\ref{thm:limsup_L} hard in general and is why, in Theorems~\ref{thm:upper_fun_C'_post_min} and~\ref{thm:lower_fun_C'}, we are required to assume that $X$ lies in the domain of attraction of an $\alpha$-stable process.
\item[(b)] If $Y$ has stationary increments (making $Y$ a subordinator), the proof of~\cite[Thm~4.1]{MR1746300} follows from Theorem~\ref{thm:limsup_L} by finding an appropriate function $f$ and sequences $(\theta_n)_{n\in\N}$ (done in~\cite[Lem.~4.2 \& 4.3]{MR1746300}) satisfying the assumptions of Theorem~\ref{thm:limsup_L}. In this case, the function $f$ is given in terms of the single-parameter Laplace exponent $\Psi_1$, see details in~\cite[Thm~4.1]{MR1746300}.\qedhere
\end{itemize}
\end{remark}

\begin{proof}[Proof of Theorem~\ref{thm:limsup_L}]
(a) Since $L$ is the right-inverse of $Y$, we have $\{L_{t_n} > f(t_n)\} = \{t_n \ge Y_{f(t_n)}\}$ for $n \in \N$. Using Chernoff's bound (Markov's inequality), we obtain 
\[
\p\big(t_n \ge Y_{f(t_n)}\big)
\le e^{\theta_n t_n}\E\big[\exp\big(-\theta_n Y_{f(t_n)}\big)\big]
=\exp(\theta_n t_n -\Psi_{f(t_n)}(\theta_n)), \quad \text{ for all }n \ge 1.
\] 
The assumption $\sum_{n=1}^\infty \exp(\theta_n t_n-\Psi_{f(t_n)}(\theta_n))<\infty$ thus implies $\sum_{n=1}^\infty \p (L_{t_n}>f(t_n))<\infty$. Hence, the Borel--Cantelli lemma 
yields $\limsup_{n \to \infty} L_{t_n}/f(t_n)\le 1$ a.s. Since $L$ is non-decreasing and $(t_n)_{n\in \N}$ is decreasing monotonically to zero, we have 
\begin{equation*}
\limsup_{t \da 0} \frac{L_t}{f(t)}
\le\limsup_{n \to\infty}\sup_{t\in[t_{n+1},t_n]} \frac{L_{t_n}}{f(t)}
\le \limsup_{n \to \infty} \frac{L_{t_n}}{f(t_n)}\cdot\limsup_{n \to\infty} \frac{f(t_n)}{f(t_{n+1})}  
\le \limsup_{n \to\infty} \frac{f(t_n)}{f(t_{n+1})} \quad \text{ a.s.,}
\end{equation*}
which gives (a).

(b) It suffices to establish that the following limits hold: $\liminf_{n\to\infty}(Y_{f(t_n)}-Y_{f(t_{n+1})})/t_n\le 1$ a.s. and $\limsup_{n\to\infty}Y_{f(t_{n+1})}/t_n\le \delta$ a.s. for any $\delta>0$. Indeed, by taking $\delta\da 0$ along a countable sequence, the second limit gives  $\limsup_{n\to\infty}Y_{f(t_{n+1})}/t_n=0$ a.s. and hence $\liminf_{n \to \infty} Y_{f(t_n)}/ t_n\le 1$ a.s. For any $t>0$ with $Y_{f(t)}\le t$ we have $L_t>f(t)$. Since the former holds for arbitrarily small values of $t>0$ a.s., we obtain $\limsup_{t \da 0}L_t/f(t)\ge 1$ a.s.

We will prove that $\liminf_{n\to\infty}(Y_{f(t_n)}-Y_{f(t_{n+1})})/t_n\le 1$ a.s. and  $\limsup_{n\to\infty}Y_{f(t_{n+1})}/t_n\le \delta$ a.s. for any $\delta>0$, using the Borel--Cantelli lemmas.
Applying Markov's inequality, we obtain the upper bound $\p(Y_t> s)\le(1-e^{-\theta s})^{-1}\E[1-e^{-\theta Y_t}]$ for all $t,s,\theta>0$, implying
\begin{equation*}
\p\big(Y_{f(t_n)}\le t_n\big) 
\ge \frac{\exp(-\Psi_{f(t_n)}(\theta_n))-\exp(-\theta_n t_n)}
    {1-\exp(-\theta_n t_n)}, \qquad \text{ for all }n \ge 1.
\end{equation*}
Since $\theta_nt_n=\theta_n\phi(\theta_n)\to \infty$ as $n \to \infty$, the denominator of the lower bound in the display above tends to $1$ as $n \to \infty$, and hence the assumption $\sum_{n=1}^\infty [\exp(-\Psi_{f(t_n)}(\theta_n))-\exp(-\theta_n t_n)]=\infty$ implies $\sum_{n=1}^\infty \p (Y_{f(t_n)}<t_n)=\infty$. Since $Y$ has non-negative independent increments and 
\begin{equation*}
\sum_{n=1}^\infty \p(Y_{f(t_n)}-Y_{f(t_{n+1})}<t_n) 
\ge\sum_{n=1}^\infty\p(Y_{f(t_n)}< t_n)=\infty,
\end{equation*} 
the second Borel--Cantelli lemma yields $\liminf_{n\to\infty}(Y_{f(t_n)}-Y_{ f(t_{n+1})})/ t_n \le 1$ a.s. 

To prove the second limit, use Markov's inequality and the elementary bound $1-e^{-x}\le x$ to get
\begin{equation*}
\p\big(Y_{f(t_{n+1})}> \delta t_n\big)
\le \frac{\E[1-\exp(-\theta_n Y_{ f(
t_{n+1})})]}{1-\exp(-\delta\theta_n t_n)}  
= \frac{1-\exp(-\Psi_{ f(t_{n+1})}(\theta_n))}{1-\exp(-\delta\theta_n t_n)}
\le \frac{\Psi_{ f(t_{n+1})}(\theta_n)}{1-\exp(-\delta\theta_n t_n)},
\end{equation*}
for all $n \in \N$. Again, the denominator tends to $1$ as $n\to\infty$ and the assumption $\sum_{n = 1}^\infty\Psi_{f(t_{n+1})}(\theta_n)<\infty$ implies $\sum_{n=1}^\infty\p(Y_{ f(t_{n+1})}> \delta t_n)<\infty$. The Borel--Cantelli lemma implies $\limsup_{n\to\infty}Y_{f(t_{n+1})}/t_n\le \delta$ a.s. and completes the proof.
\end{proof}

\subsection{Lower functions of {\nf{\emph{L}}}}
\label{subsec:lower_func_gen}
To describe the lower fluctuations of $L$, it suffices to describe the upper fluctuations of $Y$. The following result extends known results for subordinators (see, e.g.~\cite[Thm 1]{MR210190}). Given a continuous increasing function $h$ with $h(0)=0$ and $h(1)=1$, consider the following statements, used in the following result to describe the upper fluctuations of~$Y$: 
\begin{gather}
\label{eq:h_limsup}
\limsup_{t\da0}Y_t/h(t)=0,\quad \text{a.s.},\\
\label{eq:h_leq_io}
\limsup_{t\da0}Y_t/h(t)<1, \quad \text{a.s.},\\
\label{eq:Pi_large}
\Pi(\{(t,x)\,:\,t\in(0,1],\,x\ge h(t)\})<\infty,\\
\label{eq:Pi_var}
\int_{(0,1]\times(0,1)}\frac{x^2}{h(t)^2}\1_{\{2h(t)>x\}}
    \Pi(\D t,\D x)<\infty,\\
\label{eq:Pi_mean}
2^n\int_{(0,h^{-1}(2^{-n})]\times (0,2^{-n})} x\1_{\{2h(t)>x\}} \Pi(\D t,\D x) \da 0, 
\quad\text{ as }n\to\infty, \quad \text{ and }\\
\label{eq:Pi_mean_var}
\int_{(0,1]\times(0,1)}\frac{x}{h(t)}\1_{\{2h(t)>x\}}
    \Pi(\D t,\D x)<\infty.
\end{gather}

\begin{theorem}\label{thm:Y_limsup}
Let $h$ be continuous and increasing with $h(0)=0$ and $h(1)=1$. Then the following implications hold: 
\begin{multicols}{3}
\begin{itemize}[leftmargin=2em, nosep]
    \item[\nf{(a)}] \eqref{eq:h_limsup}$\implies$\eqref{eq:h_leq_io}$\implies$\eqref{eq:Pi_large},
    \item[\nf{(b)}] \eqref{eq:Pi_large}--\eqref{eq:Pi_mean}$\implies$\eqref{eq:h_limsup},
    \item[\nf{(c)}] \eqref{eq:Pi_mean_var}$\implies$\eqref{eq:Pi_var}--\eqref{eq:Pi_mean}.
\end{itemize}
\end{multicols}
\end{theorem}

\begin{figure}
\centering
\begin{tikzpicture}[node distance=1cm and 3.4cm]
\node[rectangle,draw] (a) {\eqref{eq:h_limsup}};
\node[rectangle,draw] (b)[below=of a]{\eqref{eq:h_leq_io}};
\node[rectangle,draw] (c)[right=of a]{\eqref{eq:Pi_mean}};
\node[rectangle,draw] (d)[below=of c]{\eqref{eq:Pi_var}};
\node[rectangle,draw] (e)[below=of d]{\eqref{eq:Pi_large}};
\node[rectangle,draw] (f)[right=of c]{\eqref{eq:Pi_mean_var}};
\node[rectangle,draw] (g)[right=of e]{\eqref{eq:Pi_var_inv}};
\node[rectangle,draw] (h)[right=of g]{\eqref{eq:Pi_mean_inv}};
\node[rectangle,draw] (FS)[right=of f]{\eqref{eq:Pi_mean_var_inv}};

\draw [decorate, decoration = {brace, raise=10pt, amplitude=10pt, mirror},thick] (4.15,0.2) -- (4.15,-3.2);
\draw[-implies,double equal sign distance] (a) -- (b);
\draw[implies-,double equal sign distance] (c) -- (f);
\draw[implies-,double equal sign distance] (d) -- (f);
\draw[-implies,double equal sign distance] (b) .. controls +(down:20mm) and +(left:20mm) .. (e);
\draw[implies-,double equal sign distance] (a)--(3.45,-1.5);
\draw[implies-,double equal sign distance] (d) -- (g) node[pos=0.5,above,sloped]{$h^{-1}$ concave};
\draw [decorate, decoration = {brace, raise=10pt, amplitude=10pt, mirror},thick] (4,-3.15) -- (13.1,-3.15);
\draw[-implies,double equal sign distance] (FS)--(g)node[pos=0.5,above,sloped]{$h^{-1}$ concave};
\draw[-implies,double equal sign distance] (FS)--(h)node[pos=0.5,above,sloped]{$h^{-1}$ concave};
\draw[implies-,double equal sign distance] (a) .. controls (-1.6,-3.9) .. (8.55,-3.9)node[pos=0.8,below,sloped]{$h^{-1}$ concave};
\draw[-,thick] (3.45,-1.5) -- (d);
\draw[-,thick] (8.55,-3.85) -- (g);
\end{tikzpicture}
\caption{A graphical representation of the implications in Theorem~\ref{thm:Y_limsup} and Proposition~\ref{prop:Y_limsup}.} \label{fig:overview_technical_results}
\end{figure}
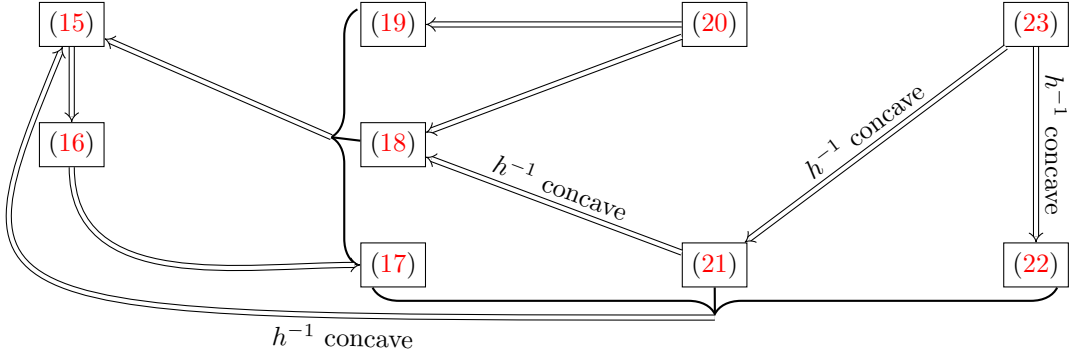

\begin{remark}\label{rem:Y_limsup}\nf
If $h$ is as in Theorem~\ref{thm:Y_limsup} and $\Pi(\{(t,x)\,:\,t\in(0,1],\,x\ge ch(t)\})=\infty$ for all $c>0$, then it follows from the negation of Theorem~\ref{thm:Y_limsup}(a) that $\limsup_{\da0}Y_t/h(t)=\infty$ a.s.
\end{remark}

In the description of the lower fluctuations of $L$, we are typically given the function $h^{-1}$ directly instead of $h$. In those cases, the conditions in Theorem~\ref{thm:Y_limsup} may be hard to verify directly (see e.g. the proof of Theorem~\ref{thm:C'_limsup}(i)). To alleviate this issue, we introduce alternative conditions describing the upper fluctuations of $Y$ in terms of the function $h^{-1}$. However, this requires the additional assumption that $h^{-1}$ is concave, see Proposition~\ref{prop:Y_limsup} below. Consider the following conditions on $h^{-1}$:

\begin{gather}
\label{eq:Pi_var_inv}
\int_{(0,1]\times(0,1)}\frac{h^{-1}(x)^2}{t^2}\1_{\{2t\ge h^{-1}(x)\}}\Pi(\D t, \D x)<\infty,\\
\label{eq:Pi_mean_inv}
2^n\int_{(0,2^{-n}]\times(0,h(2^{-n}))}h^{-1}(x)\1_{\{2t\ge h^{-1}(x)\}}\Pi(\D t, \D x)\da 0,
\quad\text{as }n\to\infty, \quad \text{ and }\\
\label{eq:Pi_mean_var_inv}
\int_{(0,1]\times(0,1)}\frac{h^{-1}(x)}{t}\1_{\{2t\ge h^{-1}(x)\}}\Pi(\D t,\D x)<\infty.
\end{gather}

\begin{proposition}
\label{prop:Y_limsup}
Let $h$ be convex and increasing with $h(0)=0$ and $h(1)=1$. Then the following implications hold: 
\begin{multicols}{3}
\begin{itemize}[leftmargin=2em, nosep]
    \item[\nf{(a)}] \eqref{eq:Pi_var_inv}$\implies$\eqref{eq:Pi_var},
    \item[\nf{(b)}] \eqref{eq:Pi_mean_var_inv}$\implies$\eqref{eq:Pi_var_inv}--\eqref{eq:Pi_mean_inv},
    \item[\nf{(c)}] \eqref{eq:Pi_large} and \eqref{eq:Pi_var_inv}--\eqref{eq:Pi_mean_inv}$\implies$\eqref{eq:h_limsup}.
\end{itemize}
\end{multicols}
\end{proposition}

The relation between the assumptions of Theorem~\ref{thm:Y_limsup} and Proposition~\ref{prop:Y_limsup} (concerning $h$ and $h^{-1}$) is described in Figure~\ref{fig:overview_technical_results}. The following elementary result explains how the upper fluctuations of $Y$ (described by Theorem~\ref{thm:Y_limsup}) are related to the lower fluctuations of $L$. 

\begin{lemma}
\label{lem:Y_vs_L}
Let $h$ be a continuous increasing function with $h(0)=0$ and denote by $h^{-1}$ its inverse. Then the following implications hold for any $c>0$:
\begin{itemize}[leftmargin=2em, nosep]
    \item[\nf{(a)}] $\liminf_{t\da0}L_t/h^{-1}(t/c)>1
    \implies\limsup_{t\da0}Y_t/h(t)\le c$,
    \item[\nf{(b)}] $\limsup_{t\da0}Y_t/h(t)< c
    \implies\liminf_{t\da0}L_t/h^{-1}(t/c)\ge 1$.
\end{itemize}
\end{lemma}

\begin{proof}
The result follows from the implications $L_u>t\implies u\ge Y_t\implies L_u\ge t$ for any $t,u>0$. Indeed, if $\liminf_{u\da0}L_u/h^{-1}(u/c)>1$ then $L_u>h^{-1}(u/c)$ for all sufficiently small $u>0$ implying that $Y_t\le ch(t)$ for all sufficiently small $t>0$ and hence $\limsup_{t\da 0}Y_t/h(t)\le c$. This establishes part (a). Part (b) follows along similar lines.
\end{proof}

A combination of Lemma~\ref{lem:Y_vs_L}, Theorem~\ref{thm:Y_limsup}, Proposition~\ref{prop:Y_limsup} and Remark~\ref{rem:Y_limsup} yield the following corollary.

\begin{corollary}
\label{cor:L_liminf}
Let $h$ be a continuous and increasing function with $h(0)=0$ and $h(1)=1$ such that $\lim_{c\da0}\limsup_{t\da0}h^{-1}(ct)/h^{-1}(t)=0$. Then the following results hold:
\begin{itemize}[leftmargin=2.5em, nosep]
    \item[{\nf(i)}] If $\liminf_{t\da0}L_t/h^{-1}(t/c)> 1$ a.s. for some $c\in(0,1)$ then~\eqref{eq:Pi_large} holds.
    \item[{\nf(ii)}] Suppose~\eqref{eq:Pi_large}--\eqref{eq:Pi_mean} hold, then $\liminf_{t\da0}L_t/h^{-1}(t)=\infty$ a.s.
    \item[{\nf(ii')}] Suppose $h$ is convex and conditions~\eqref{eq:Pi_large} and~\eqref{eq:Pi_var_inv}--\eqref{eq:Pi_mean_inv} hold, then $\liminf_{t\da0}L_t/h^{-1}(t)=\infty$ a.s.
    \item[{\nf(iii)}] If $\Pi(\{(t,x)\,:\,t\in(0,1],\,x\ge ch(t)\})=\infty$ for all $c>0$ then $\liminf_{t\da0}L_t/h^{-1}(t)=0$ a.s.
\end{itemize}
\end{corollary}

To prove Theorem~\ref{thm:Y_limsup} we require the following lemma. For all $t\ge 0$ denote by $\Delta_t\coloneqq Y_t-Y_{t-}$ the jump of $Y$ at time $t$, so that $Y_t=\sum_{u\le t}\Delta_u$ since $Y$ is a pure-jump additive process. We also let $N$ denote the Poisson jump measure of $Y$, given by $N(A)\coloneqq |\{t:(t,\Delta_t)\in A\}|$ for $A\subset[0,\infty)\times (0,\infty)$ and note that its mean measure is $\Pi(\D t,\D x)$. 

\begin{lemma}\label{lem:upper_func_Y}
Let $h$ be continuous and increasing with $h(0)=0$ and $h(1)=1$. Assume \eqref{eq:Pi_large}--\eqref{eq:Pi_mean} hold, then 
$\limsup_{t \da 0}Y_t/h(t)
=\limsup_{t \da 0}Y_{h^{-1}(t)}/t=0$ a.s.
\end{lemma}

\begin{proof}
For all $n \in \N$, we let $B_n\coloneqq [2^{-n},\infty)$ and set $C_n\coloneqq h^{-1}((2^{-n-1},2^{-n}])\times B_n$. Then we have 
\begin{equation*}
\sum_{n \in \N}\p(N(C_n)\geq 1)
= \sum_{n \in \N}\big(1-e^{-\Pi(C_n)}\big)
\le \sum_{n \in \N} \Pi(C_n), 
\end{equation*}
by the definition of $N$ and the inequality $1-e^{-x}\le x$. Note that $\sum_{n \in \N} \Pi(C_n)<\infty$ by~\eqref{eq:Pi_large}, since 
\[
\sum_{n\in\N} \Pi(C_n)
\le\Pi(\{(t,x)\,:\, t\in[0,1],\, x\ge h(t)\})<\infty.
\]
By the Borel--Cantelli lemma, there exists some $n_0\in\N$ with $N(h^{-1}((2^{-n-1},2^{-n}])\times B_n)=0$ a.s. for all $n \geq n_0$. By the mapping theorem, the random measure $N_h(A\times B)\coloneqq N(h^{-1}(A)\times B)$ for any measurable $A,B\subset[0,\infty)$, is a Poisson random measure with mean measure $\Pi_h(A\times B)\coloneqq\Pi(h^{-1}(A),B)$. Note that $Y_{h^{-1}(t)}
=\int_{(0,h^{-1}(t)]\times (0,\infty)}xN(\D u,\D x)
=\int_{(0,t]\times (0,\infty)}xN_h(\D u,\D x)$ for $t\ge 0$ and, for any $n \geq n_0$ and $t \in (2^{-n-1},2^{-n}]$, we have $|Y_{h^{-1}(t)}/t|\le\zeta_n\coloneqq 2^{n+1}\sum_{m=n}^\infty\xi_m$, where
\begin{equation*}
\xi_m\coloneqq \int_{(2^{-m-1},2^{-m}]\times (0,2^{-m})}
    x N_h(\D u,\D x),
\qquad m\in\N.
\end{equation*}
To complete the proof, it suffices to show that $\zeta_n\da 0$ a.s. as $n\to\infty$. Fubini's theorem yields 
\begin{align*}
2^{-n-1}\E[\zeta_n]
&=\sum_{m=n}^\infty\int_{(2^{-m-1},2^{-m}]\times (0,2^{-m})} x\Pi_h(\D u,\D x)\\
&=\int_{(0,2^{-n}]\times (0,2^{-n})} 
    x\sum_{m=n}^\infty \1_{\{x<2^{-m}\}}\1_{\{u\le 2^{-m}<2u\}} 
        \Pi_h(\D u,\D x)\\
&\le \int_{(0,2^{-n}]\times (0,2^{-n})} 
    x\1_{\{2u>x\}} \Pi_h(\D u,\D x)
=\int_{(0,h^{-1}(2^{-n})]\times (0,2^{-n})} 
    x\1_{\{2h(v)>x\}} \Pi(\D v,\D x).
\end{align*} 
By assumption~\eqref{eq:Pi_mean}, we deduce that $\E[\zeta_n]\da0$ as $n\to\infty$. Similarly, note that
\begin{equation*}
\Var(\zeta_n)
= 4^{n+1}\sum_{m= n}^\infty\int_{(2^{-m-1},2^{-m}]\times (0,2^{-m})} x^2 \Pi_h(\D u,\D x),
\end{equation*} 
and hence, by Fubini's theorem and assumption~\eqref{eq:Pi_var}, we have
\begin{align*}
\sum_{n=1}^\infty\Var(\zeta_n)
&= \sum_{m=1}^\infty\sum_{n=1}^m 4^{n+1}\int_{(2^{-m-1},2^{-m}]\times (0,2^{-m})} x^2 \Pi_h(\D t,\D x)\\
&\le\sum_{m=1}^\infty 4^{m+2}\int_{(2^{-m-1},2^{-m}]\times (0,2^{-m})} x^2 \Pi_h(\D u,\D x)\\
&=\int_{(0,1]\times (0,1)}x^2\sum_{m=1}^\infty 4^{m+2}\1_{\{x<2^{-m}\}}\1_{\{u<2^{-m}<2u\}} \Pi_h(\D u,\D x)\\
&\le 4^2\int_{(0,1]\times(0,1)}\frac{x^2}{u^2}\1_{\{2u>x\}}\Pi_h(\D u,\D x)\\
&=4^2\int_{(0,h^{-1}(1)]\times(0,1)}\frac{x^2}{h(v)^2}\1_{\{2h(v)>x\}}\Pi(\D v,\D x)<\infty.
\end{align*} 
Thus, we find that the sum $\sum_{n=1}^\infty(\zeta_n-\E[\zeta_n])^2$ has finite mean equal to $\sum_{n=1}^\infty\Var(\zeta_n)<\infty$ and is thus finite a.s. Hence, the summands must tend to $0$ a.s. and, since $\E[\zeta_n]\to 0$, we deduce that $\zeta_n \da 0$ a.s. as $n \to \infty$. 
\end{proof}

\begin{proof}[Proof of Theorem~\ref{thm:Y_limsup}] 
It is obvious that~\eqref{eq:h_limsup} implies~\eqref{eq:h_leq_io}. If~\eqref{eq:h_leq_io} holds, then $Y_t<h(t)$ for all sufficiently small $t$. Thus, the path bound $Y_t\ge\Delta_t$ implies $\p(N(\{(t,x)\,:\,t \in [0,1],\,x>h(t)\})<\infty)=1$ and hence~\eqref{eq:Pi_large}. By Lemma~\ref{lem:upper_func_Y},
conditions~\eqref{eq:Pi_large}--\eqref{eq:Pi_mean} imply~\eqref{eq:h_limsup}, so it remains to show that~\eqref{eq:Pi_mean_var} implies~\eqref{eq:Pi_var} and~\eqref{eq:Pi_mean}.

It is easy to see that~\eqref{eq:Pi_mean_var} implies~\eqref{eq:Pi_var}. Moreover, if~\eqref{eq:Pi_mean_var} holds, then 
\begin{multline*}
2^n\int_{(0,h^{-1}(2^{-n})]\times (0,2^{-n})} x\1_{\{2h(t)>x\}} \Pi(\D t,\D x)\\
\le \int_{(0,h^{-1}(1)]\times (0,1)} \frac{x}{h(t)}\1_{\{2h(t)>x\}}
\1_{{(0,h^{-1}(2^{-n})]\times (0,2^{-n})}}(t,x) 
\Pi(\D t,\D x),
\end{multline*}
where the upper bound is finite for all $n\in\N$ and tends to $0$ as $n\to\infty$ by the monotone convergence theorem, implying~\eqref{eq:Pi_mean}. 
\end{proof}

\begin{proof}[Proof of Proposition~\ref{prop:Y_limsup}]
Since $h^{-1}$ is concave with $h^{-1}(0)=0$, then $x\mapsto h^{-1}(x)/x$ is decreasing, so the condition $h(t)>x/2$ implies $(x/2)/h(t)\le h^{-1}(x/2)/t\le h^{-1}(x)/t$. The inequality $h^{-1}(x)/x\le h^{-1}(x/2)/(x/2)$ implies that $\{(t,x)\,:\,2h(t)>x\}\subset\{(t,x)\,:\,2t>h^{-1}(x)\}$, proving the first claim: \eqref{eq:Pi_var_inv} implies~\eqref{eq:Pi_var}. 

Since $h^{-1}$ is concave with $h^{-1}(0)=0$, it is subadditive, implying 
\[
\zeta_t \coloneqq  \sum_{u\le t}h^{-1}(\Delta_u)\ge h^{-1}(Y_t).
\]
Since $\limsup_{t\da 0}\zeta_t/t\le c$ implies $\limsup_{t\da 0}Y_t/h(ct)\le 1$ for $c>0$ and $h$ is a convex function, it suffices to show that $\limsup_{t\da 0}\zeta_t/t=0$ a.s. Note that $\zeta$ is an additive process with jump measure $\Pi(\D t, h(\D x))$. Applying Theorem~\ref{thm:Y_limsup} to $\zeta$ with the identity function yields the result, completing the proof. 
\end{proof}

\begin{remark}\label{rem:stationary_increments}
We now show that, when the increments of $Y$ are stationary (making $Y$ a subordinator), Theorem~\ref{thm:Y_limsup} gives a complete characterisation of the upper functions of $Y$, recovering~\cite[Thm~1]{MR210190} (see also~\cite[Prop.~4.4]{MR1746300}). This is done in two steps.

Suppose $h$ is convex and $Y$ has stationary increments with mean jump measure $\Pi(\D t, \D x)=\Pi((0,1],\D x)\D t$. Then $h^{-1}$ is concave and the additive process $\wt Y_t\coloneqq \sum_{s\le t}h^{-1}(\Delta_s)\ge h^{-1}(Y_t)$ has mean jump measure $\Pi(\D t,h(\D x))$, making it a subordinator. Theorem~\ref{thm:Y_limsup} applied to $\wt Y$ and the identity function makes all conditions~\eqref{eq:Pi_large}--\eqref{eq:Pi_mean}  equivalent to $\int_{(0,1)}h^{-1}(x)\Pi((0,1],\D x)<\infty$ and therefore, by Theorem~\ref{thm:Y_limsup}, also equivalent to the condition $\limsup_{t\da0}\wt Y_t/t=0$ a.s. 

Note that condition~\eqref{eq:Pi_large} for $\wt Y$ and the identity function coincides with condition~\eqref{eq:Pi_large} for $Y$ and $h$. This equivalence, together with the fact that the limit $\limsup_{t\da 0}\wt Y_t/t=0$ implies $\limsup_{t\da 0}Y_t/h(t)=0$, shows that both limits are either $0$ a.s. or positive a.s. jointly. Thus, $\limsup_{t\da0}Y_t/h(t)=0$ a.s. if and only if $\int_{(0,1)}h^{-1}(x)\Pi((0,1],\D x)<\infty$ and, if the latter condition fails, then $\limsup_{t\da0} Y_t/h(t)=\infty$ a.s. by Remark~\ref{rem:Y_limsup}. This is precisely the criterion given in~\cite[Thm~1]{MR210190} (see also~\cite[Prop.~4.4]{MR1746300}). 
\end{remark}

Remark~\ref{rem:stationary_increments} shows that condition~\eqref{eq:Pi_large} perfectly describes the upper fluctuations of $Y$ when $Y$ has stationary increments, making  conditions~\eqref{eq:Pi_var} \&~\eqref{eq:Pi_mean} appear superfluous. These conditions are, however, not superfluous since~\eqref{eq:Pi_large} by itself cannot fully characterise the upper fluctuations of $Y$, as the following example shows. 

\begin{example}\label{ex:breakequivalence}
Let $\Pi(\D t,\D x) = \sum_{n\in\N}n^{-1}2^n\delta_{(2^{-n},2^{-n}/n)}(\D t,\D x)$, where $\delta_x$ denotes the Dirac measure at $x$, and consider the corresponding additive process $Y$ (whose existence is ensured by~\cite[Thm~15.4]{MR1876169}). Since $\p(\xi\ge\mu)\ge 1/5$ for every Poisson random variable $\xi$ with mean $\mu\ge 2$~\cite[Eq.~(6)]{pelekis2017lower}, we get $\sum_{n\in\N}\p(N(\{(2^{-n},2^{-n}/n)\})\ge 2^n/n)=\infty$. The second Borel--Cantelli lemma then shows that $\Delta_{2^{-n}}\ge 1/n^2$ i.o. Thus, $Y_{2^{-n}}/2^{-n}\ge 2^{n}\Delta_{2^{-n}}\ge 2^n/n^2$ i.o., implying $\limsup_{t\da0}Y_t/t=\infty$ a.s. even when condition~\eqref{eq:Pi_large} holds. In fact, $\Pi(\{(t,x)\,:\,t\in(0,1],\,x\ge ct\})<\infty$ for all $c>0$. 
\end{example}

\section{The vertex time process and the proofs of the results in Section~\ref{sec:small-time-derivative}}
\label{sec:proofs}

We first recall basic facts about the vertex time process $\tau=(\tau_s)_{s \in \R}$. Fix a deterministic time horizon $T>0$, let $C$ be the convex minorant of $X$ on $[0,T]$ with right-derivative $C'$ and recall the definition $\tau_s=\inf\{t>0:C'_t>s\}$ for any slope $s\in\R$. By the convexity of $C$, the right-derivative $C'$ is non-decreasing and right-continuous, making $\tau$ a non-decreasing right-continuous process with $\lim_{s\to-\infty}\tau_s=0$ and $\lim_{s\to\infty}\tau_s=T$. Intuitively put, the process $\tau$ finds the times in $[0,T]$ at which the slopes increase as we advance through the graph of the convex minorant $t\mapsto  C_t$ chronologically. We remark that the vertex time process can be constructed directly from $X$ without any reference to the convex minorant $C$, as follows (cf.~\cite[Thm~11.1.2]{MR1739699}): for each slope $s\in\R$ and time epoch $t\ge 0$, define  
$X^{(s)}_t\coloneqq X_t-st$, 
$\un X^{(s)}_t\coloneqq \inf_{u\in[0,t]}X_u^{(s)}$
and note
$\tau_s= \sup\big\{t\in[0,T]\,:\, X_{t-}^{(s)}\wedge X_t^{(s)}=\un X_T^{(s)}\big\}$, where $X_{u-}^{(s)}\coloneqq \lim_{v\uparrow u}X_{u-}^{(s)}$ for $u>0$ and $X_{0-}^{(s)}\coloneqq X_{0}^{(s)}=0$. Put differently, subtracting a constant drift~$s$ from the L\'evy process $X$ ``rotates'' the convex hull so that the vertex time $\tau_s$ becomes the time the minimum of $X^{(s)}$ during the time interval $[0,T]$ is attained.

\subsection{The vertex time process over exponential times}
\label{subsec:tau_exp}
Fix any $\lambda>0$ and let $E$ be an independent exponential random variable with unit mean. Let $\wh C\coloneqq(\wh C_t)_{t \in [0,E/\lambda]}$ be the convex minorant of $X$ over the exponential time-horizon $[0,E/\lambda]$ and denote by $\wh\tau$ the right-continuous inverse of $\wh C'$, i.e. $\wh\tau_s\coloneqq \inf\{u\in [0,E/\lambda]:\wh C'_u>s\}$ for $s\in\R$. Hence, in the remainder of the paper, the processes with (resp. without) a `hat' will refer to the processes whose definition is based on the path of $X$ on $[0,E/\lambda]$ (resp. $[0,T]$), where $E$ is an exponential random variable with unit mean independent of $X$ and $T>0$ is fixed and deterministic.

It is more convenient to consider the vertex time processes over an independent exponential time horizon rather than the fixed time horizon $T$, as this does not affect the small-time behaviour of the process (see Corollary~\ref{cor:trivial} below), while making its law more tractable. Moreover, as we will see, to analyse the fluctuations of $\wh C'$ over short intervals, it suffices to study those of $\wh\tau$. By~\cite[Cor.~3.2]{fluctuation_levy}, the process $\wh\tau$ has independent but non-stationary increments and its Laplace exponent is given by
\begin{equation}
\label{eq:cf_tau}
\E[e^{-u \wh \tau_s}]
=e^{-\Phi_s(u)}, \quad \text{ where } \quad \Phi_s(u)\coloneqq \int_0^\infty (1-e^{-u t})e^{-\lambda t}\p(X_t\le st)\frac{\D t}{t},
\end{equation}
for all $u \ge 0$ and $s\in\R$. The following lemma states that, after a vertex time, the convex minorants $C$ and $\wh C$ must agree for a positive amount of time, see Figure~\ref{fig:CM_agree} for a pictorial description.


\begin{lemma}
\label{lem:CM_agree}
For any $s\in\mL^+(\mS)$,  on the event $\{\tau_s<E/\lambda\le T\}$,  we have $\tau_s=\wh\tau_s$ and the convex minorants $C$ and $\wh C$ agree on and interval $[0,\tau_s+m]$ for a random $m>0$. If $X$ is of infinite variation, the functions $C$ and $\wh C$ agree on an interval $[0,m]$ for a random variable $m$
satsifying $0<m\le\min\{T,E/\lambda\}$ a.s.
\end{lemma}

Since the L\'evy process $X$ and the exponential time $E$ are independent, 
$\p(\tau_s<E/\lambda\le T)>0$.

\begin{proof}
The proof follows directly from the definition of the convex minorant of $f$ as the greatest convex function dominated by the path of $f$ over the corresponding interval. Let $f$ be a measurable function on $[0,t]$ with piecewise linear convex minorant $M^{(t)}$. Then, for any vertex time $v\in(0,t)$ of $M^{(t)}$ and any $u\in(v,t]$, the convex minorant $M^{(u)}$ of $f$ on $[0,u]$ equals $M^{(t)}$ over the interval $[0,v]$. The result then follows since the condition $s\in\mL^+(\mS)$ (resp. $X$ has infinite variation) implies that there are infinitely many vertex times immediately after $\tau_s$ (resp. $0$).
\end{proof}

\begin{figure}
    \centering
    \includegraphics[width=0.65\textwidth]{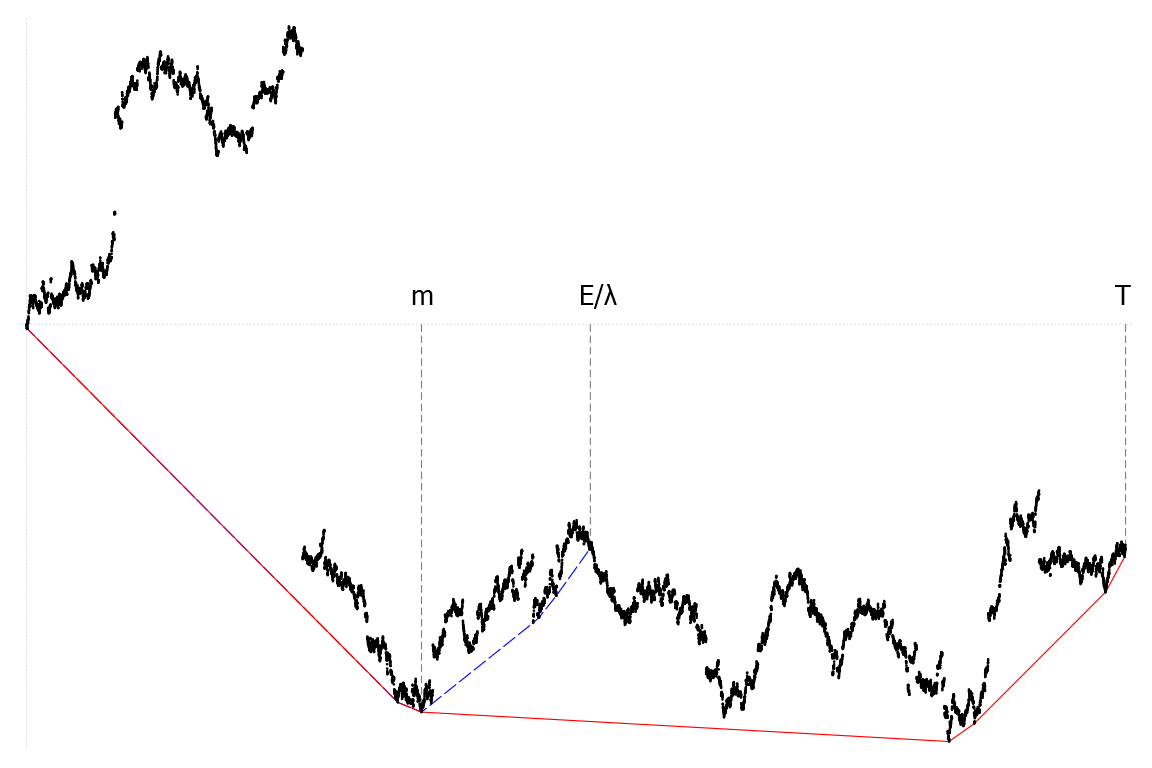}
    \caption{The picture shows a path of $X$ (black) and its convex minorants $C$ (red) on $[0,T]$ and $\wh C$ (blue) on $[0,E/\lambda]$. Both convex minorants agree until time $m$, after which they may behave very differently.}
    \label{fig:CM_agree}
\end{figure}

The following result shows that local properties of $C$ agree with those of $\wh C$. Multiple extensions are possible, but we opt for the following version as it is simple and sufficient for our purpose.

\begin{corollary}\label{cor:trivial}
Fix any measurable function $f:(0,\infty) \to (0,\infty)$.\\
{\nf{(a)}} If $s\in\mL^+(\mS)$, then the following limits are a.s. constants on $[0,\infty]$:
\[
\limsup_{t\da 0}\frac{C'_{t+\tau_s}-s}{f(t)}=\limsup_{t\da 0}\frac{\wh C'_{t+\wh \tau_s}-s}{f(t)}
\quad \text{and}\quad
\liminf_{t\da 0}\frac{C'_{t +\tau_s}-s}{f(t)}=\liminf_{t\da 0}\frac{\wh C'_{t+\wh \tau_s}-s}{f(t)}.
\]
{\nf{(b)}} If $X$ is of infinite variation, then the following limits are a.s. constants on $[0,\infty]$:
\[
\limsup_{t\da 0}C'_t/f(t)=\limsup_{t\da 0}\wh C'_t/f(t)
\qquad\text{and}\qquad
\liminf_{t\da 0}C'_t/f(t)=\liminf_{t\da 0}\wh C'_t/f(t).
\]
\end{corollary}

\begin{proof}
We will prove part (a) for $\liminf$, with the remaining proofs being analogous.
First note that the assumption $s\in\mL^+(\mS)$ implies that  $(\tau_{u+s}-\tau_s)_{u\ge 0}$ and the additive processes $(\wh\tau_{u+s}-\wh\tau_s)_{u\ge 0}$ have infinite activity as $u\da 0$. Thus, applying Blumenthal's 0--1 law~\cite[Cor.~19.18]{MR1876169} to $(\wh\tau_{u+s}-\wh\tau_s)_{u\ge 0}$ (and using the fact that $\wh C'_{\wh\tau_s}=s$ a.s.), implies that $\liminf_{t\da 0}(\wh C'_{t+\wh\tau_s}-s)/f(t)$ is a.s. equal to some constant $\mu$ in $[0,\infty]$. Moreover, by the independence of the increments of $\wh\tau_s$, this limit holds even when conditioning on the value of $\wh\tau_s$. Recall further that $\wh\tau_s=\tau_s$ on the event $\{\tau_s<E/\lambda\le T\}$ by Lemma~\ref{lem:CM_agree}.
By Lemma~\ref{lem:CM_agree} and the independence of $E$ and $X$, we a.s. have
\begin{align*}
0<\p(\tau_s<E/\lambda\le T\,|\,\tau_s)
&=\p\Big(\liminf_{t\da 0}(\wh C'_{t+\wh\tau_s}-s)/f(t)=\mu,
    \,\tau_s<E/\lambda\le T\,\Big|\,\tau_s\Big)\\
&=\p\Big(\liminf_{t\da 0}(C'_{t+\tau_s}-s)/f(t)=\mu,
    \,\tau_s<E/\lambda\le T\,\Big|\,\tau_s\Big)\\
&=\p\Big(\liminf_{t\da 0}(C'_{t+\tau_s}-s)/f(t)=\mu\,\Big|\,\tau_s\Big)
    \p(\tau_s<E/\lambda\le T\,|\,\tau_s),
\end{align*}
implying that $\liminf_{t\da 0}(C'_{t+\tau_s}-s)/f(t)=\mu$ a.s.
\end{proof}

By virtue of Corollary~\ref{cor:trivial} it suffices to prove all the results in Section~\ref{sec:small-time-derivative} for $\wh C$ instead of $C$. This allows us to use the independent increment structure of the right inverse $\wh\tau$ of the right-derivative $\wh C'$. 

\begin{example}[Cauchy process]
\label{ex:Cauchy}
If $X$ is a Cauchy process, then the Laplace exponent of $\wh\tau_u$ factorises $\Phi_u(w)=\p(X_1\le u)\int_0^\infty(1-e^{-wt})e^{-\lambda t}t^{-1}\D t$ for any $u\in\R$ and $w\ge 0$. This implies that $\wh\tau$ has the same law as a gamma subordinator time-changed by the distribution function $u\mapsto\p(X_1\le u)=\frac{1}{2}+\frac{1}{\pi}\arctan(cu+\mu)$ for some $c>0$ and $\mu=\tan(\pi(\frac{1}{2}-\rho))$. This result can be used as an alternative to~\cite[Thm~2]{MR1747095}, in conjunction with classical results on the fluctuations of a gamma process (see, e.g.~\cite[Ch.~4]{MR1746300}), to establish~\cite[Cor.~3]{MR1747095} and all the other results in~\cite{MR1747095}. 
\end{example}

The proofs of the results in Section~\ref{sec:small-time-derivative} are based on the results of Section~\ref{sec:additive}: we will construct a non-decreasing additive process $Y=(Y_t)_{t\ge 0}$, started at $0$, in terms of $\wh\tau$ and apply the results in Section~\ref{sec:additive} to $Y$ and its inverse $L=(L_u)_{u\ge 0}$. These proofs are given in the following subsections.

\subsection{Upper and lower functions at time \texorpdfstring{$\tau_s$}{tau} - proofs}
\label{subsec:proofs_c'_tau_s}

Let $s\in\mL^+(\mS)$. Fix any $\lambda>0$ and let $Y_u\coloneqq\wh\tau_{u+s}-\wh \tau_s$, $u\ge 0$. Then the right-inverse $L_u\coloneqq\inf\{t>0\,:\,Y_t>u\}$ of $Y$ equals $L_u=\wh C'_{u+\wh\tau_s}-s$ for $u\ge 0$. Note that $Y$ has independent increments and~\eqref{eq:cf_tau} implies
\begin{equation}
\label{eq:psi_defn_post_min}
\Psi_u(w)
\coloneqq-\log \E[e^{-wY_u}]
=\int_0^\infty (1-e^{-wt})\Pi((0,u],\D t),
\quad\text{for all $w,u \ge 0$},
\end{equation}
where $\Pi(\D u,\D t)=e^{-\lambda t}\p((X_t-st)/t \in \D u)t^{-1}\D t$ is the mean jump measure of $Y$. 

\begin{proof}[Proof of Theorem~\ref{thm:post-min-lower}]
Since $s\in\mL^+(\mS)$, all three parts of the result follow from a direct application of Proposition~\ref{prop:Y_limsup} and Corollary~\ref{cor:L_liminf} to the definitions of $Y$ and $L$ above.
\end{proof}

To prove Theorem~\ref{thm:upper_fun_C'_post_min}, we require the following two lemmas. The first lemma establishes some general regularity for the densities of $X_t$ as a function of $t$ and the second lemma provides a strong asymptotic control on the function $\Psi_s(u)$ as $s\da 0$ and $u\to\infty$. Recall that, when $X$ is of finite variation, $\gamma_0=\lim_{t\da 0}X_t/t$ denotes the natural drift of $X$. 

\begin{lemma}\label{lem:generalized_Picard}
Let $X \in \mZ_{\alpha,\rho}$ for some $\alpha \in (0,1)$ and $\rho \in (0,1]$ and denote by $g$ its normalising function.\\ 
{\nf{(a)}} Define $Q_t\coloneqq (X_t-\gamma_0t)/g(t)$, then $Q_t$ has an infinitely differentiable density $p_t$ such that $p_t$ and each of its derivatives $p_t^{(k)}$ are uniformly bounded: $\sup_{t\in(0,1]}\sup_{x\in\R}|p_t^{(k)}(x)|<\infty$ for any $k\in\N\cup\{0\}$.\\
{\nf{(b)}} Define $\wt Q_t\coloneqq X_t/\sqrt{t}$, then $\wt Q_t$ has an infinitely differentiable density $\tilde p_t$ such that $\tilde p_t$ and each of its derivatives $\tilde p_t^{(k)}$ are uniformly bounded: $\sup_{t\in[1,\infty)}\sup_{x\in\R}|\tilde p_t^{(k)}(x)|<\infty$ for any $k\in\N\cup\{0\}$.
\end{lemma}

For two functions $f_1,f_2:(0,\infty)\to(0,\infty)$ we say $f_1(t)\sim f_2(t)$ as $t\da 0$ if $\lim_{t\da 0}f_1(t)/f_2(t)=1$. 

\begin{proof}[Proof of Lemma~\ref{lem:generalized_Picard}]
Part (a). We assume without loss of generality that $g(t)\le 1$ for $t\in(0,1]$, and note that $Q_t$ is infinitely divisible. Denote by $\nu_{Q_t}$ the L\'evy measure of $Q_t$, and note for $A \subset \R$ that $\nu_{Q_t}(A)=t\nu(g(t)A)$ and
\[
\ov \sigma^2_{Q_t}(u)
\coloneqq \int_{(-u,u)}x^2 \nu_{Q_t}(\D x)
=\frac{t}{g(t)^2}\int_{(-ug(t),ug(t))}x^2 \nu(\D x)=\frac{t}{g(t)^2}\ov\sigma^2(ug(t)),
\] for $t \in (0,1]$ and $u \in \R\setminus \{0\}$. The regular variation of $\ov\nu$ (see~\cite[Thm~2]{MR3784492}), Fubini's theorem and Karamata's theorem~\cite[Thm~1.5.11(ii)]{MR1015093} imply that, as $u\da0$,
\begin{align*}
\ov\sigma^2(u)
&=-\int_0^u x^2\ov\nu(\D x)
=-\int_0^u 2\int_0^x z\D z\ov\nu(\D x)
=-\int_0^u\int_z^u 2 z\ov\nu(\D x)\D z\\
&=\int_0^u 2 z(\ov\nu(z)-\ov\nu(u)) \D z
=\int_0^u 2 z\ov\nu(z) \D z-u^2\ov\nu(u)
\sim \frac{\alpha}{2-\alpha}u^2\ov\nu(u).
\end{align*}

Since $X\in\mZ_{\alpha,\rho}$, \cite[Thm~2]{MR3784492} implies that $g^{-1}(u)u^{-2}\ov\sigma^2(u)\to c_0$ for some $c_0>0$ as $u\da 0$. Thus, 
\begin{equation*}
0<\inf_{z\in(0,1]}\frac{g^{-1}(z)}{z^2}\ov\sigma^2(z)
\le \inf_{u,t\in(0,1]}\frac{g^{-1}(ug(t))}{u^2g(t)^2}\ov\sigma^2(ug(t)).
\end{equation*}
Since $g$ is regluarly varying with index $1/\alpha$, we suppose that $g(t)=t^{1/\alpha}\varpi(t)$ for a slowly varying function $\varpi$. Thus, Potter's bounds~\cite[Thm~1.5.6]{MR1015093} imply that, for some constant $c>1$ and all $t,u\in(0,1]$, we have  $\varpi(t)/\varpi(tu^\beta)\le cu^{-\beta\delta}$ for $\delta=1/\beta-1/\alpha>0$. Hence, we obtain $ug(t)\le cg(tu^\beta)$ and moreover $g^{-1}(ug(t))\le c^{\beta} tu^\beta$ for all $t\in(0,1]$ and $u\in(0,1/c]$. Multiplying the rightmost term on the display above (before taking infimum) by $tu^\beta/g^{-1}(ug(t))$ gives 
\begin{equation}
\label{eq:picard_condition}
\inf_{t\in(0,1]}\inf_{u\in(0,1/c]} u^{\beta-2}\ov\sigma_{Q_t}^2(u)
=\inf_{t\in(0,1]}\inf_{u\in(0,1/c]}\frac{tu^{\beta}}
    {u^2g(t)^2}\ov\sigma^2(ug(t))
>0.
\end{equation}
Hence,~\cite[Lem.~2.3]{picard_1997} gives the desired result.

Part (b). As before, we see that $\ov\sigma^2_{\wt Q_t}(u)=\ov\sigma^2(u\sqrt{t})$. Hence, the left side of~\eqref{eq:picard_condition} gives
\[
\inf_{t\in[1,\infty)}\inf_{u\in(0,1]}u^{\beta-2}\ov\sigma^2_{\wt Q_t}(u)
=\inf_{u\in(0,1]}u^{\beta-2}\ov\sigma^2(u)>0,
\]
for any $\beta\in(0,\alpha)$. Thus,~\cite[Lem.~2.3]{picard_1997} gives the desired result.
\end{proof}

\begin{lemma} \label{lem:asymp_equiv_Psi_domstable_post_min}
Let $X\in\mZ_{\alpha,\rho}$ for some $\alpha\in(0,1)$ and $\rho\in(0,1]$, denote by $g$ its normalising function and define $G(t)=t/g(t)$ for $t>0$. The following statements hold for any sequences $(u_n)_{n \in \N} \subset (0,\infty)$ and $(s_n)_{n \in \N}\subset (0,\infty)$ such that $u_n\to\infty$ and $s_n\da 0$ as $n \to \infty$:
\begin{itemize}[leftmargin=2.5em, nosep]
\item[\nf(i)]~if $u_nG^{-1}(s_n^{-1})\to \infty$, then 
$\Psi_{s_n}(u_n) \sim \rho\log(u_nG^{-1}(s_n^{-1}))$,
\item[\nf(ii)]~if $u_nG^{-1}(s_n^{-1})\to 0$, then 
$\Psi_{s_n}(u_n)=\Oh([u_nG^{-1}(s_n^{-1})]^q+s_n)$ for any $q\in (0,1]$ with $q<1/\alpha-1$.
\end{itemize}
\end{lemma}

\begin{proof}
Part (i). Define $Q_t\coloneqq (X_t-\gamma_0t)/g(t)$ and note that
\[
\Psi_{s_n}(u_n)
=\int_0^\infty (1-e^{-tu_n})e^{-\lambda t} 
    \p\big(0<Q_t \leq s_nG(t) \big)\frac{\D t}{t}, \quad \text{for all }n \in \N.
\]
Fix $\delta\in(0,\rho/3)$, let $\kappa_n\coloneqq G^{-1}(\delta/s_n)$ and note that $\kappa_n\da 0$ as $n\to\infty$. We will now split the integral in the previous display at $\kappa_n$ and $1$ and find the asymptotic behaviour of each of the resulting integrals.

The integral on $[1,\infty)$ is bounded as $n\to\infty$: 
\begin{align*}
\int_1^\infty (1-e^{-tu_n})e^{-\lambda t} 
    \p\big(0<Q_t \leq s_nG(t) \big)\frac{\D t}{t} 
\le \int_1^\infty e^{-\lambda t}\frac{\D t}{t}<\infty.
\end{align*} 
Next, we consider the integral on $[\kappa_n,1)$. By Lemma~\ref{lem:generalized_Picard}(a), there exists a uniform upper bound $C>0$ on the densities of $Q_t$, $t\in(0,1]$. An application of~\cite[Thm~1.5.11(i)]{MR1015093} gives, as $n\to\infty$,
\begin{align*}
\int_{\kappa_n}^1 (1-e^{-u_nt})e^{-\lambda t} 
    \p\big(0<Q_{t} \le s_n G(t)\big)\frac{\D t}{t}
&\le C\int_{\kappa_n}^1 s_n G(t) \frac{\D t}{t} \sim \frac{\alpha C}{1-\alpha} s_nG(\kappa_n)
= \frac{\delta\alpha C}{1-\alpha}<\infty.
\end{align*}
Since we will prove that $\Psi_{s_n}(u_n)\to\infty$ as $n\to\infty$, the asymptotic behaviour of $\Psi_{s_n}(u_n)$ will be driven by asymptotic behaviour of the integral on $(0,\kappa_n)$:
\begin{equation}
\label{eq:J_tn}
J_n^0
\coloneqq \int_0^1 (1-e^{-u_n\kappa_nt})e^{-\lambda \kappa_nt} 
    \p\big(0<Q_{\kappa_n t} \leq s_nG(\kappa_nt)\big)\frac{\D t}{t}.
\end{equation}

We will show that, asymptotically as $n \to \infty$, we may replace the probability in the integrand with the probability $\p(0<Z<\delta t^{1-1/\alpha})$ in terms of the limiting $\alpha$-stable random variable $Z$. Since $Z$ has a bounded density (see, e.g.~\cite[Ch.~4]{MR1745764}), the weak convergence $Q_t \cid Z$ as $t \da 0$ implies that the distributions functions converge in Kolmogorov distance by~\cite[1.8.31--32, p.~43]{MR1353441}. Thus, since $\kappa_n\to0$ as $n\to\infty$, there exists some $N_\delta\in\N$ such that 
\[
\sup_{n\ge N_\delta}\sup_{t\in(0,\kappa_n]}\sup_{x\in\R}
    |\p(0<Q_t\le x)-\p(0<Z\le x)|<\delta,
\]
where $\delta\in(0,\rho/3)$ is as before, arbitrary but fixed. In particular, the following inequality holds $\sup_{n\ge N_\delta}\sup_{t\in(0,\kappa_n]}
    |\p(0<Q_t\le s_n G(t))-\p(0<Z\le s_nG(t))|<\delta$.
For any $N\ge N_\delta$ the triangle inequality yields
\begin{align*}
B_{\delta,N}
&\coloneqq\sup_{n\ge N}\sup_{t\in(0,1]}
|\p(0<Z<\delta t^{1-1/\alpha}) 
    - \p(0<Q_{t\kappa_n}\le s_n G(t\kappa_n))|\\
&\le \delta + \sup_{n\ge N}\sup_{t\in(0,1]}
|\p(0<Z<\delta t^{1-1/\alpha}) 
    - \p(0<Z\le s_n G(t\kappa_n))|\\
&\le \delta + \sup_{n\ge N}\sup_{t\in(0,1]}
\p(m_{t,n}<Z<M_{t,n}),
\end{align*} 
where $m_{t,n}\coloneqq\min\{s_n G(t\kappa_n),\delta t^{1-1/\alpha}\}$ and $M_{t,n}\coloneqq\max\{s_n G(t\kappa_n),\delta t^{1-1/\alpha}\}$. We aim to show that $B_{\delta,N_\delta'}<2\delta$ for some $N_\delta'\in\N$. 

By~\cite[Ch.~4]{MR1745764}, there exists $K>0$ such that the stable density of $Z$ is bounded by the function $x\mapsto Kx^{-\alpha-1}$ for all $x>0$. Thus, since $M_{t,n}-m_{t,n}=|\delta t^{1-1/\alpha} - s_nG(t\kappa_n)|$, we have
\begin{equation}
\label{eq:min_max_Z}
\begin{split}
\p(m_{t,n}<Z<M_{t,n})
&\le Km_{t,n}^{-\alpha-1} 
    |\delta t^{1-1/\alpha} - s_nG(t\kappa_n)|\\
&\le K((\delta t^{1-1/\alpha})^{-\alpha-1} + (s_nG(t\kappa_n))^{-\alpha-1})
    |\delta t^{1-1/\alpha} - s_nG(t\kappa_n)|.
\end{split}
\end{equation}
To show that this converges uniformly in $t\in(0,1]$, we consider both summands. First, we have
\[
(\delta t^{1-1/\alpha})^{-\alpha-1} 
    |\delta t^{1-1/\alpha} - s_nG(t\kappa_n)|\\
=\delta^{-\alpha}
\bigg| t^{1-\alpha} 
    - \frac{(t\kappa_n)^{(1-\alpha^2)/\alpha}G(t\kappa_n)}
        {\kappa_n^{(1-\alpha^2)/\alpha}G(\kappa_n)}\bigg|,
\]
which tends to $0$ as $n\to\infty$ uniformly in $t\in(0,1]$ by~\cite[Thm~1.5.2]{MR1015093} since $t\mapsto t^{(1-\alpha^2)/\alpha}G(t)$ is regularly varying at $0$ with index $1-\alpha>0$ (recall that $g$ is regularly varying at $0$ with index $1/\alpha$ and $G(t)=t/g(t)$). Similarly, since $s_n=\delta /G(\kappa_n)$, we have
\[
(s_nG(t\kappa_n))^{-\alpha-1} 
    |\delta t^{1-1/\alpha} - s_nG(t\kappa_n)|\\
=\delta^{-\alpha}
\bigg| \frac{(t\kappa_n)^{1-1/\alpha}
        G(t\kappa_n)^{-\alpha-1}}{\kappa_n^{1-1/\alpha}G(\kappa_n)^{-\alpha-1}} 
    - \frac{G(t\kappa_n)^{-\alpha}}{G(\kappa_n)^{-\alpha}}\bigg|.
\]
Since both terms in the last line converge to $\delta^{\alpha}t^{1-\alpha}$ as $n\to\infty$ uniformly in $t\in(0,1]$ by~\cite[Thm~1.5.2]{MR1015093}, the difference tends to $0$ uniformly too. Hence, the right side of~\eqref{eq:min_max_Z} converges to $0$ as $n\to\infty$ uniformly in $t\in(0,1]$. Thus, for a sufficiently large $N'_\delta$, we have
\begin{equation}
\label{eq:B_delta}
\sup_{n\ge N'_\delta}\sup_{t\in(0,1]}
|\p(0<Z<\delta t^{1-1/\alpha}) 
    - \p(0<Q_{t\kappa_n}\le s_n G(t\kappa_n))|
=B_{\delta,N'_\delta}<2\delta.
\end{equation}

We now analyse a lower bound on the integral $J_n^0$ in~\eqref{eq:J_tn}. By~\eqref{eq:B_delta}, for all $n \ge N_\delta'$, we have
\[
J_n^0
\ge \int_0^{1} (1-e^{-u_n\kappa_nt})e^{-\lambda\kappa_nt} 
    \big(\p\big(0<Z \le \delta t^{1-1/\alpha}\big)-2\delta\big)\frac{\D t}{t}.
\]
Recall that $\kappa_n=G^{-1}(\delta/s_n)$, define $\xi_n\coloneqq G^{-1}(1/s_n)$ and note from the regular variation of $G^{-1}$ that $\kappa_n/\xi_n\to \delta^{\alpha/(\alpha-1)}$ as $n\to\infty$, implying  $\log(u_n\kappa_n)\sim\log(u_n\xi_n)$ as $n \to \infty$ since $u_n\xi_n\to\infty$. We split the integral from the display above at $\log(u_n\kappa_n)^{-1}$ and note that
\begin{multline*}
\int_{\log(u_n\kappa_n)^{-1}}^1 (1-e^{-u_n\kappa_nt})e^{-\lambda\kappa_nt} 
\big(\p\big(0<Z \le \delta t^{1-1/\alpha}\big)+2\delta\big)\frac{\D t}{t}\\ \le \big(1+2\delta\big)\int_{\log(u_n\kappa_n)^{-1}}^1
\frac{\D t}{t}
=\big(1+2\delta\big)\log (\log(u_n\kappa_n))
\sim\big(1+2\delta\big)\log (\log(u_n\xi_n)),
\quad\text{as }n\to\infty.
\end{multline*} 
For the integral over $(0,\log(u_n\kappa_n)^{-1})$, first note that, for all sufficiently large $n\in\N$, we have
\[
\p(0<Z \le \delta t^{1-1/\alpha})
\ge \p(0<Z \le \delta \log(u_n\kappa_n)^{1/\alpha-1})
>\rho-\delta,
\qquad t\in(0,\log(u_n\kappa_n)^{-1}),
\] 
since $u_n\kappa_n\to\infty$. Thus, we have
\begin{multline*}
    \int_0^{\log(u_n\kappa_n)^{-1}} (1-e^{-u_n\kappa_nt})e^{-\lambda\kappa_nt} 
    \big(\p\big(0<Z \le \delta t^{1-1/\alpha}\big)-2\delta\big)\frac{\D t}{t} \\
    \ge \big(\rho-3\delta\big)e^{-\lambda\kappa_n/\log(u_n\kappa_n)}\int_0^{\log(u_n\kappa_n)^{-1}}(1-e^{-u_n\kappa_nt}) 
    \frac{\D t}{t}\sim \big(\rho-3\delta\big)\log(u_n\xi_n), \quad \text{ as }n \to \infty,
\end{multline*}
where the asymptotic equivalence follows from the fact that
$u_n\kappa_n/\log(u_n\kappa_n)\to\infty $ as $n\to\infty$
and
$\int_0^1 (1-e^{-xt})t^{-1}\D t\sim \log x$ as $x\to\infty$. (In fact, we have $\int_0^1 (1-e^{-xt})t^{-1}\D t= \log x+\Gamma(0,x)+\gamma$ for $x>0$ where $\Gamma(0,x)=\int_x^\infty t^{-1}e^{-t}\D t$ is the upper incomplete gamma function and $\gamma$ is the Euler--Mascheroni constant.) This shows that $\liminf_{n\to\infty}J_n^0/\log(u_n\xi_n)\ge \rho-3\delta>0$ since $\delta\in(0,\rho/3)$. 

Similarly,~\eqref{eq:B_delta} implies that for all $n\ge N'_\delta$, we have
\begin{align*}
J_n^0
&\le \int_0^{1} (1-e^{-u_n\kappa_nt})e^{-\lambda\kappa_nt} 
    \big(\p\big(0<Z \le \delta t^{1-1/\alpha}\big)+2\delta\big)
        \frac{\D t}{t}\\
&\le (\rho+2\delta)\int_0^{1} (1-e^{-u_n\kappa_nt})
        \frac{\D t}{t}
\sim (\rho+2\delta)\log(u_n\xi_n),
\quad\text{as }n\to\infty,
\end{align*}
implying $\limsup_{n\to\infty}J_n^0/\log(u_n\xi_n)\le \rho+2\delta$. Altogether, we deduce that
\[
\rho-3\delta
\le \liminf_{n\to\infty}\Psi_{s_n}(u_n)/\log(u_n\xi_n)
\le\limsup_{n\to\infty}\Psi_{s_n}(u_n)/\log(u_n\xi_n)
\le \rho+2\delta.
\]
Since $\delta\in(0,\rho/3)$ is arbitrary and the sequence $\Psi_{s_n}(u_n)/\log(u_n\xi_n)$ does not depend on $\delta$, we may take $\delta\da 0$ to obtain Part (i).

Part (ii). We will bound each of the terms in $\Psi_{s_n}(u_n)=J_n^1+J_n^2+J_n^3$, where $\xi_n=G^{-1}(1/s_n)$ and
\begin{gather*}
J_n^1\coloneqq
\int_0^{\xi_n} (1-e^{-u_nt})e^{-\lambda t} \p(0<Q_t\leq s_n G(t))\frac{\D t}{t}, 
\quad
J_n^2\coloneqq
\int_{\xi_n}^{1} (1-e^{-u_nt})e^{-\lambda t} \p(0<Q_t\leq s_n G(t))\frac{\D t}{t},\\
\text{and}\qquad
J_n^3\coloneqq
\int_{1}^\infty (1-e^{-u_nt})e^{-\lambda t} \p(0<X_t-\gamma_0t\leq s_n t)\frac{\D t}{t}.
\end{gather*} 

Recall that our assumption in part (ii) states that $u_n\xi_n\to 0$ as $n\to\infty$. Using the elementary inequality $1-e^{-x}\le x$ for $x \ge 0$, we obtain $J_n^1=\Oh(u_n\xi_n)$ as $n \to \infty$. Next we bound $J_n^3$.  Lemma~\ref{lem:generalized_Picard}(b) shows the existence of a uniform upper bound $\wt C>0$ on the densities of $X_t/\sqrt{t}$. Thus it holds that $\p(0<X_t-\gamma_0t\le s_n t)=\p(\gamma_0\sqrt{t}<X_t/\sqrt{t}\le (\gamma_0+s_n) \sqrt{t})
\le \wt C s_n\sqrt{t}$ and hence 
\begin{equation*}
J_n^3
\le\wt C s_n\int_1^\infty 
    t^{-1/2}e^{-\lambda t}\D t
=\Oh(s_n), \quad \text{ as }n \to \infty.
\end{equation*}

It remains to bound $J_n^2$. Let $q\in(0,1]$ with $q<1/\alpha-1$ and $C>0$ be a uniform bound on the densities of $Q_t$ (whose existence is guaranteed by Lemma~\ref{lem:generalized_Picard}(a)). The elementary bound $1-e^{-x}\le x^q$ for $x\ge 0$ for $q\in (0,1]$ and~\cite[Thm~1.5.11(i)]{MR1015093} yield
\[
J_n^2
\le C u_n^qs_n \int_{\xi_n}^1t^q G(t) \frac{\D t}{t} 
\sim \frac{C}{1/\alpha-q-1} u_n^q s_n G(\xi_n) \xi_n^q 
= \Oh (u_n^q\xi_n^q ), \quad \text{ as }n \to \infty.\qedhere
\]
\end{proof}

\begin{proof}[Proof of Theorem~\ref{thm:upper_fun_C'_post_min}]
Throughout this proof we let $\phi(u)\coloneqq \gamma u^{-1} (\log\log u)^r$, for some $\gamma>0$, $r \in \R$. 

Part (i). Since $p$ is arbitrary on $(1/\rho,\infty)$ and $f(t)=1/G(t\log^p(1/t))$, it suffices to show that $\limsup_{t\da0}(\wh C'_{t+\wh\tau_s}-s)/f(t)=\limsup_{t \downarrow 0} L_t/f(t)<\infty$ a.s. (Recall that $L_t=C'_{t+\wh\tau_s}-s$ and $\Psi_u(w)=-\log\E[e^{-wY_u}]$ for all $u,w\ge 0$.) By Theorem~\ref{thm:limsup_L}(a), it suffices to find a positive sequence $(\theta_n)_{n \in \N}$ with $\lim_{n\to \infty}\theta_n=\infty$ such that  $\sum_{n=1}^\infty \exp(\theta_n t_n-\Psi_{f(t_n)}(\theta_n))<\infty$ and $\limsup_{n \to \infty} f(t_n)/f(t_{n+1})<\infty$ where $t_n\coloneqq \phi(\theta_n)$.

Let $\theta_n\coloneqq e^n$ and $r=0$. Note that the regular variation of $f$ at $0$ yields $\limsup_{n \to \infty}f(t_n)/f(t_{n+1})=\lim_{n \to \infty}f(t_n)/f(t_{n+1})= e^{1-1/\alpha}$.
Thus, it suffices to prove that the series above is finite. Since $t_n=\phi(\theta_n)$, it follows that $t_n\theta_n=\gamma$. Note from the definition of $f$ that 
\begin{equation}
\label{eq:stable_f_asymp4_pm}
uG^{-1}(f(\phi(u))^{-1})
=u h(u)(\log (\phi(u)^{-1}))^{p}
=\gamma (\log (\gamma^{-1} u))^{p}
\sim\gamma (\log u)^{p}\to\infty,
\quad \text{ as }u \to \infty.
\end{equation}
By Lemma~\ref{lem:asymp_equiv_Psi_domstable_post_min}(i) we have $\Psi_{f(t_n)}(\theta_n)\sim \rho\log(\theta_n G^{-1}(f(t_n)^{-1}))$ as $n \to \infty$, since $\theta_n G^{-1}(f(t_n)^{-1}) \sim
\gamma(\log\theta_n)^{p}\to \infty$ as $n \to \infty$ by~\eqref{eq:stable_f_asymp4_pm}. 

Fix some $\ve>0$ with $(1-\ve)\rho p>1$. Note that  $\Psi_{f(t_n)}(\theta_n)
\ge (1-\ve)\rho p\log\log\theta_n$ for all sufficiently large $n$. It suffices to show that the following sum is finite:
\[
\sum_{n=1}^\infty \exp\big(\gamma
-(1-\ve)\rho p\log\log\theta_n\big).
\]
Since $(1-\ve)\rho p>1$, the sum in the display above is bounded by a multiple of $\sum_{n=1}^\infty n^{-(1-\ve)\rho p}<\infty$. 

Part (ii). As before, since $p$ is arbitrary in $(0,1/\rho)$, it suffices to show that $\limsup_{t \downarrow 0} L_t/f(t)\ge 1$ a.s. By Theorem~\ref{thm:limsup_L}(b), it suffices to find a positive sequence $(\theta_n)_{n \in \N}$ satisfying $\lim_{n\to \infty}\theta_n=\infty$, such that $\sum_{n=1}^\infty (\exp(-\Psi_{f(t_n)}(\theta_n))-\exp(-\theta_n t_n))=\infty$ and $\sum_{n=1}^\infty \Psi_{f(t_{n+1})}( \theta_n)<\infty$. 

Let $r=\gamma=1$, choose $\sigma>1$ and $\ve>0$ to satisfy $\sigma(1+\ve)\rho p<1$ and set $\theta_n\coloneqq e^{n^\sigma}$ for $n \in \N$. We start by showing that the second sum in the paragraph above is finite. Since $\sigma>1$,~\eqref{eq:stable_f_asymp4_pm} yields
\begin{equation}
\label{eq:stable_f_asymp3_pm}
\theta_nG^{-1}(f(t_{n+1})^{-1})\sim \frac{\theta_n}{\theta_{n+1}}(\log\theta_{n+1})^{p}\log\log\theta_{n+1}\da 0, \quad \text{ as }n \to \infty.
\end{equation} 
Hence, Lemma~\ref{lem:asymp_equiv_Psi_domstable_post_min}(ii) with $q \in (0,1]$ and $q<1/\alpha-1$ and~\eqref{eq:stable_f_asymp3_pm} imply 
\begin{equation*}
\Psi_{f(t_{n+1})}(\theta_n)
=\Oh\big([\theta_nG^{-1}(f(t_{n+1})^{-1})]^{q}+f(t_{n+1})\big),
\quad\text{as }n\to\infty.
\end{equation*} 
By~\eqref{eq:stable_f_asymp3_pm}, it is enough to show that
\begin{align*}
\sum_{n=1}^\infty \bigg(\frac{\theta_n}{\theta_{n+1}}(\log\theta_{n+1})^{p}\log\log\theta_{n+1}\bigg)^{q}<\infty,
\qquad\text{and}\qquad
\sum_{n=1}^\infty f(t_{n+1})<\infty.
\end{align*}
Newton's generalised binomial theorem implies that $\theta_n/\theta_{n+1}=\exp(n^\sigma-(n+1)^\sigma)\le \exp(-\sigma n^{\sigma-1}/2)$ for all sufficiently large $n$. Since $\log\theta_{n+1}\sim n^\sigma$, we conclude that the first series in the previous display is indeed finite. The second series is also finite since $f\circ h$ is regularly varying at infinity with index $(\alpha-1)/\alpha<0$ (recall that $t_{n+1}=\phi(\theta_{n+1})$). 

Next we prove that $\sum_{n=1}^\infty (\exp(-\Psi_{f(t_n)}(\theta_n))-\exp(-\theta_n t_n))=\infty$. Note that we have $\exp(-\theta_nt_n)=\exp(-\log \log\theta_n)=n^{-\sigma}$, which is summable. Applying Lemma~\ref{lem:asymp_equiv_Psi_domstable_post_min}(i) and~\eqref{eq:stable_f_asymp4_pm}, we see that $\Psi_{f(t_n)}(\theta_n)\sim \rho\log(\theta_n G^{-1}(f(t_n)^{-1}))$ as $n \to \infty$. As in Part~(i), it is easy to see that for every $\ve>0$, the inequality $\Psi_{f(t_n)}(\theta_n)
\le (1+\ve)\rho
    p\log\log\theta_n$
holds for all sufficiently large $n$. Thus 
$\exp(-\Psi_{f(t_{n})}(\theta_n))
\ge n^{-\sigma(1+\ve)\rho p}$ is not summable (since $\sigma(1+\ve)\rho p<1$): $\sum_{n=1}^\infty \exp(-\Psi_{f(t_n)}(\theta_n))=\infty$, completing the proof.
\end{proof}

\subsection{Upper and lower functions at time \texorpdfstring{$0$}{0} - proofs}
\label{subsec:proofs_c'_at_0}

Fix any $\lambda>0$. Let $Y_s \coloneqq  \tau_{-1/s}$ for $s\in(0,\infty)$ and note that the mean jump measure of $Y_s$ is given by 
\begin{equation*}
\Pi(\D s, \D t)
\coloneqq t^{-1}e^{-\lambda t}\p(-t/X_t\in \D s)\D t,
\end{equation*} 
implying $\Pi((0,s], \D t)
=t^{-1}e^{-\lambda t}\p(X_t\le -t/s)\D t$. Since $\wh C'$ is the right-inverse of $\wh\tau$, we have the identity $\wh C'_t=-1/L_t$ where $L_t\coloneqq\inf\{s>0\,:\, Y_s>t\}$. Thus, $\limsup_{t\da0}|\wh C'_t| f(t)$ equals $0$ (resp. $\infty$) if and only if $\liminf_{t\da 0}L_t/f(t)$ equals $\infty$ (resp. $0$). Corollary~\ref{cor:L_liminf} and Proposition~\ref{prop:Y_limsup} above are the ingredients in the proof of Theorem~\ref{thm:C'_limsup}.

\begin{proof}[Proof of Theorem~\ref{thm:C'_limsup}]
Since the conditions in Theorem~\ref{thm:Y_limsup} only involve integrating the mean measure $\Pi$ of $Y$ near the origin, we may ignore the factor $e^{-\lambda t}$ in the definition of the mean measure $\Pi$ above. After substituting $\Pi(\D u, \D t)
= t^{-1}\p(-t/X_t\in \D u)\D t$ in conditions \eqref{eq:Pi_large} and \eqref{eq:Pi_var_inv}--\eqref{eq:Pi_mean_inv}, we obtain the conditions in~\eqref{eq:C'_large}--\eqref{eq:C'_mean}.
Thus, Corollary~\ref{cor:L_liminf} and the identity $\wh C_t=-1/L_t$ yield the claims in Theorem~\ref{thm:C'_limsup}. 
\end{proof}

The following technical lemma which establishes the asymptotic behaviour of the characteristic exponent $\Phi$ defined in~\eqref{eq:cf_tau}. This result plays an important role in the proof of Theorem~\ref{thm:lower_fun_C'}. We will assume that $X \in\mZ_{\alpha,\rho}$. For simplicity, by virtue of~\cite[Eq.~(1.5.1)~\&~Thm~1.5.4]{MR1015093}, we assume without loss of generality that: $g(t)=1$ for $t\ge 1$, $g$ is continuous and decreasing on $(0,1]$ and the function $G(t)=t/g(t)$ is continuous and increasing on $(0,\infty)$. Hence, the inverse $G^{-1}$ of $G$ is also continuous and increasing.

\begin{lemma} \label{lem:asymp_equiv_Phi_domstable}
Let $X\in\mZ_{\alpha,\rho}$ for some $\alpha\in(1,2]$ and $\rho\in(0,1)$ and assume $\E[X_1^2]<\infty$ and $\E[X_1]=0$. The following statements hold for any sequences $(u_n)_{n \in \N} \subset (0,\infty)$ and $(s_n)_{n \in \N}\subset \R_-$ such that $u_n\to\infty$ and $s_n\to-\infty$ as $n \to \infty$:
\begin{itemize}[leftmargin=2.5em, nosep]
\item[\nf(i)] if $u_nG^{-1}(|s_n|^{-1})\to \infty$, then 
$\Phi_{s_n}(u_n) \sim (1-\rho)\log(u_nG^{-1}(|s_n|^{-1}))$,
\item[\nf(ii)] if $u_nG^{-1}(|s_n|^{-1})\da 0$, then 
$\Phi_{s_n}(u_n)=\Oh( [u_nG^{-1}(|s_n|^{-1})]^{(\alpha-1)/2}+|s_n|^{-2})$.
\end{itemize}
\end{lemma}

\begin{proof}
Part (i). Denote $Q_t\coloneqq X_t/g(t)$ and note that, for all $n \in \N$,
\begin{align*}
\Phi_{s_n}(u_n)
&=\int_0^\infty (1-e^{-tu_n})e^{-\lambda t} 
    \p\big(Q_t \leq s_nG(t) \big)\frac{\D t}{t}.
\end{align*}
For every $\delta>0$ let $\kappa_n\coloneqq G^{-1}(\delta/|s_n|)$ and note that $\kappa_n\da 0$ as $n\to\infty$. The integral in the previous display is split at $\kappa_n$ and we control the two resulting integrals. 

We start with the integral on $[\kappa_n,\infty)$. For any $q\in(0,\alpha)$ we claim that $K\coloneqq\sup_{t\ge 0}\E[|Q_t|^q]<\infty$. Indeed, since $\E[X_t^2]<\infty$,
$t^{-1/2}g(t)Q_t$ converges weakly to a normal random variable as $t\to\infty$. Applying~\cite[Lem.~3.1]{bang2021asymptotic} gives $\sup_{t\ge 1}\E[|Q_t|^q]t^{-q/2}g(t)^q<\infty$, and hence $\sup_{t\ge 1}\E[|Q_t|^q]<\infty$ since $t^{-1/2}g(t)$ is bounded from below for $t\ge 1$. Similarly,~\cite[Lem.~4.8--4.9]{MR4161123} imply that $\sup_{t\le 1}\E[|Q_t|^q]<\infty$, and thus $K<\infty$. Markov's inequality then yields 
\begin{equation}
\label{eq:upper_bound_stable_prob}
K\ge \sup_{n\in\N}\sup_{t\ge\kappa_n}
    |s_n|^{q}G(t)^{q}
        \p(Q_t\le s_nG(t)).
\end{equation}
Let $q'\coloneqq q(1-1/\alpha)>0$ and note that $G(t)^{-q}$ is regularly varying at $0$ with index $-q'$. By~\eqref{eq:upper_bound_stable_prob} we have $\p(Q_t\le s_nG(t))\le K|s_n|^{-q}G(t)^{-q}$ for all $t\ge \kappa_n$ and $n\in\N$. Hence, Karamata's theorem~\cite[Thm~1.5.11]{MR1015093} gives
\begin{align*}
\int_{\kappa_n}^\infty (1-e^{-u_nt})e^{-\lambda t} 
    \p\big(Q_t \le s_nG(t)\big)\frac{\D t}{t}
&\le K\int_{\kappa_n}^\infty |s_n|^{-q}e^{-\lambda t}G(t)^{-q} \frac{\D t}{t}\\
&\sim \frac{K}{q'}|s_n|^{-q} G(\kappa_n)^{-q}
=\frac{K}{q'\delta^q}
<\infty,
\quad\text{as }n\to\infty.
\end{align*}
Thus, the integral $\int_{\kappa_n}^\infty (1-e^{-u_nt})e^{-\lambda t} 
    \p\big(Q_t \le s_nG(t)\big)t^{-1}\D t$ is bounded as $n\to\infty$.

It remains to establish the asymptotic growth of the corresponding integral on $(0,\kappa_n)$. Since the limiting $\alpha$-stable random variable $Z$ has a bounded density (see, e.g.~\cite[Ch.~4]{MR1745764}), the weak convergence of $Q_t \cid Z$ as $t \da 0$ extends to convergence in Kolmogorov distance by~\cite[1.8.31--32, p.~43]{MR1353441}. Thus, there exists some $N_\delta\in\N$ such that 
\[
\sup_{n\ge N_\delta}\sup_{t\in[0,\kappa_n]}
    |\p(Q_t\le s_n G(t))-\p(Z\le s_nG(t))|<\delta.
\]
Since $G(\kappa_n)=\delta/|s_n|$ and $\p(Z\le 0)=1-\rho$, 
the triangle inequality yields
\[
B_\delta
\coloneqq\sup_{n\ge N_\delta}\sup_{t\in[0,\kappa_n]}
|1-\rho - \p(Q_t\le s_n G(t))|
\le |1-\rho - \p(Z\le-\delta)|+\delta.
\]
which tends to $0$ as $\delta\da 0$.

Define $\xi_n\coloneqq G^{-1}(1/|s_n|)$ for and note from the regular variation of $G^{-1}$ that $\kappa_n/\xi_n\to \delta^{\alpha/(\alpha-1)}$ as $n\to\infty$, implying  $\log(u_n\kappa_n)\sim\log(u_n\xi_n)$ as $n \to \infty$ since $u_n\xi_n\to\infty$. As in the proof of Lemma~\ref{lem:asymp_equiv_Psi_domstable_post_min} above, we have $\int_0^1 (1-e^{-xt})t^{-1}\D t \sim \log x$ as $x\to\infty$. Since $u_n\xi_n\to\infty$ and $\xi_n\da 0$ as $n\to\infty$, we have 
\begin{align*}
\int_0^{\kappa_n} (1-e^{-u_nt})e^{-\lambda t} 
    \p\big(Q_t \le s_nG(t)\big)\frac{\D t}{t}
&\le (1-\rho+B_\delta)
    \int_0^{\kappa_n} (1-e^{-u_nt})e^{-\lambda t}
        \frac{\D t}{t}\\
&\sim(1-\rho+B_\delta)\log(u_n\xi_n), \quad \text{as }n \to \infty.
\end{align*} 
This implies that $\limsup_{n\to\infty}\Phi_{s_n}(u_n)/\log(u_n\xi_n)\le 1-\rho+B_\delta$. A similar argument can be used to obtain  $\liminf_{n\to\infty}\Phi_{s_n}(u_n)/\log(u_n\xi_n)\ge 1-\rho-B_\delta$. Since $\delta>0$ is arbitrary and $B_\delta\da 0$ as $\delta\da 0$, we deduce that $\Phi_{s_n}(u_n)\sim(1-\rho)\log(u_n\xi_n)$ as $n\to\infty$. 

Part (ii). We will bound each of the terms in $\Phi_{s_n}(u_n)=J_n^1+J_n^2+J_n^3$, where $\xi_n=G^{-1}(1/|s_n|)$ and
\begin{gather*}
J_n^1\coloneqq
\int_0^{\xi_n} (1-e^{-u_nt})e^{-\lambda t} \p(X_t\leq s_n t)\frac{\D t}{t}, 
\qquad
J_n^2\coloneqq
\int_{1}^\infty (1-e^{-u_nt})e^{-\lambda t} \p(X_t\leq s_n t)\frac{\D t}{t},\\
\text{and}\qquad
J_n^3\coloneqq
\int_{\xi_n}^{1} (1-e^{-u_nt})e^{-\lambda t} \p(Q_t\leq s_n G(t))\frac{\D t}{t}.
\end{gather*} 
The elementary inequality $1-e^{-x}\le x$ for $x\ge 0$ implies that the integrand of $J_n^1$ is bounded by $u_n$. Hence, we have $J_n^1=\Oh(u_n\xi_n)=\Oh((u_n\xi_n)^{(\alpha-1)/2})$ as $n\to\infty$. 

To bound $J_n^2$, we use Markov's inequality as follows: since $\E[X_t^2]
= \E[X_1^2]t$ for all $t>0$, we have
$\p(X_t\le s_n t)
\le \E[X_1^2]t/(|s_n|^2 t^2)
=\E[X_1^2]|s_n|^{-2} t^{-1}$, for all $n\in\N$, $t>0$.
Thus, we get
\begin{align*}
J_n^2
\le \frac{\E[X_1^2]}{|s_n|^2}\int_{1}^\infty  \frac{\D t}{t^2}
=\frac{\E[X_1^2]}{|s_n|^2}=\Oh(|s_n|^{-2}), \quad \text{ as }n \to \infty.
\end{align*}

It remains to bound $J_n^3$. Let $r\coloneqq(\alpha-1)/2$, pick any $q\in(\alpha/2,\alpha)$ and recall from Part~(i) that $K=\sup_{t\ge 0}\E[|Q_t|^q]<\infty$. Note that $q'=q(1-1/\alpha)>r$, so Karamata's theorem~\cite[Thm~1.5.11]{MR1015093}, the inequality in~\eqref{eq:upper_bound_stable_prob} and the elementary bound $1-e^{-x}\le x^r$ for $x\ge 0$ yield
\begin{align*}
J_n^3
\le Ku_n^r\int_{\xi_n}^{1} 
    t^r|s_n|^{-q} G(t)^{-q}\frac{\D t}{t}
\sim \frac{Ku_n^r}{q'-r}\xi_n^{r}|s_n|^{-q} G(\xi_n)^{-q}
= \frac{K}{q'-r}(u_n\xi_n)^r, \quad \text{ as }n \to \infty.
\end{align*}
We conclude that $J_n^3=\Oh((u_n\xi_n)^r)$ as $n\to\infty$, completing the proof.
\end{proof}

\begin{proof}[Proof of Theorem~\ref{thm:lower_fun_C'}]
Throughout this proof we let $\phi(u)\coloneqq \gamma u^{-1} (\log\log u)^r$, for some $\gamma>0$ and $r \in \R$. By Remark~\ref{rem:modify_nu} we may and do assume without loss of generality that $(X_t)_{t \geq 0}$ has a finite second moment and zero mean.

Part (i). Since $p$ is arbitrary on $(1/(1-\rho),\infty)$, it suffices to show that $\liminf_{t\da 0}|\wh C_t'|f(t)>0$ a.s. where $f(t)=G(t\log^p(1/t))$. Since $\wh C_t'=-1/L_t$, this is equivalent to $\limsup_{t \downarrow 0} L_t/f(t)<\infty$ a.s. Recall that $\Psi_u(w)=\log\E[e^{-wY_u}]=\log\E[e^{-w\wh\tau_{-1/u}}]=\Phi_{-1/u}(w)$ for all $u>0$ and $w\ge 0$. By virtue of Theorem~\ref{thm:limsup_L}(a), it suffices to show that $\sum_{n=1}^\infty \exp(\theta_n t_n-\Psi_{f(t_n)}(\theta_n))<\infty$ and $\limsup_{n \to \infty} f(t_n)/f(t_{n+1})<\infty$ for $t_n\coloneqq \phi(\theta_n)$ and a positive sequence $(\theta_n)_{n \in \N}$ with $\lim_{n\to \infty}\theta_n=\infty$.

Let $\theta_n\coloneqq e^n$ and $r=0$. Note that the regular variation of $f$ at $0$ yields $\limsup_{n \to \infty}f(t_n)/f(t_{n+1})=\lim_{n \to \infty}f(t_n)/f(t_{n+1})= e^{1-1/\alpha}$.
Thus, it suffices to prove that the series is finite. Since $t_n=\phi(\theta_n)$, it follows that $t_n\theta_n=\gamma$. Note from the definition of $f$ that 
\begin{equation}
\label{eq:stable_f_asymp4}
uG^{-1}(f(\phi(u)))
=u \phi(u)(\log (\phi(u)^{-1}))^{p}
=\gamma (\log (\gamma^{-1} u))^{p}
\sim\gamma (\log u)^{p}\to\infty,
\quad \text{ as }u \to \infty.
\end{equation}
By Lemma~\ref{lem:asymp_equiv_Phi_domstable}(i) we have $\Psi_{f(t_n)}(\theta_n)=\Phi_{-1/f(t_n)}(\theta_n)\sim (1-\rho)\log(\theta_n G^{-1}(f(t_n)))$ as $n \to \infty$, since $\theta_n G^{-1}(f(t_n)) \sim
\gamma(\log\theta_n)^{p}\to \infty$ as $n \to \infty$ by~\eqref{eq:stable_f_asymp4}. 

Fix some $\ve>0$ with $(1-\ve)(1-\rho)p>1$. Note that we have $\Psi_{f(t_n)}(\theta_n)
\ge (1-\ve)(1-\rho)p\log\log\theta_n$ for all sufficiently large $n$. It is enough to show that the following sum is finite:
\[
\sum_{n=1}^\infty \exp\big(\gamma 
-(1-\ve)(1-\rho)p\log\log\theta_n\big).
\]
Since $(1-\ve)(1-\rho)p>1$, this sum is bounded by a multiple of $\sum_{n=1}^\infty n^{-(1-\ve)(1-\rho)p}<\infty$. 

Part (ii). As before, since $p$ is arbitrary in $(0,1/(1-\rho))$, it suffices to show $\liminf_{t\da 0}|\wh C'_t|f(t)<\infty$ a.s. By Theorem~\ref{thm:limsup_L}(b), it suffices to show that there exists some $r>0$ and a positive sequence $(\theta_n)_{n \in \N}$ satisfying $\lim_{n\to \infty}\theta_n=\infty$, such that $\sum_{n=1}^\infty (\exp(-\Psi_{f(t_n)}(\theta_n))-\exp(-\theta_n t_n))=\infty$ and $\sum_{n=1}^\infty \Psi_{f(t_{n+1})}( \theta_n)<\infty$. 

Let $\gamma=r=1$, choose $\sigma>1$ and $\ve>0$ satisfying $\sigma(1+\ve)p(1-\rho)<1$ (recall $p(1-\rho)<1$) and set $\theta_n\coloneqq e^{n^\sigma}$. We start by showing that the second sum is finite. Since $\sigma>1$,~\eqref{eq:stable_f_asymp4} yields
\begin{equation}
\label{eq:stable_f_asymp3}
\theta_nG^{-1}(f(t_{n+1}))\sim \frac{\theta_n}{\theta_{n+1}}(\log\theta_{n+1})^{p}\da 0, \quad \text{ as }n \to \infty.
\end{equation} 
Hence, the time-change $\wh C'_t=-1/L_t$, Lemma~\ref{lem:asymp_equiv_Phi_domstable}(ii) and~\eqref{eq:stable_f_asymp3} imply 
\begin{align*}
\Psi_{f(t_{n+1})}(\theta_n)
=\Phi_{-1/f(t_{n+1})}(\theta_n)
&=\Oh\big([\theta_nG^{-1}(f(t_{n+1}))]^{(\alpha-1)/2}+f(t_{n+1})^2\big),
\quad\text{as }n\to\infty.
\end{align*} 
By~\eqref{eq:stable_f_asymp3}, it is enough to show that
\begin{align*}
\sum_{n=1}^\infty \bigg(\frac{\theta_n}{\theta_{n+1}}(\log\theta_{n+1})^{p}\log\log\theta_{n+1}\bigg)^{(\alpha-1)/2}<\infty,
\qquad\text{and}\qquad
\sum_{n=1}^\infty f(t_{n+1})^2<\infty.
\end{align*}
Newton's generalised binomial theorem implies that $\theta_n/\theta_{n+1}=\exp(n^\sigma-(n+1)^\sigma)\le \exp(-\sigma n^{\sigma-1}/2)$ for all sufficiently large $n$. Since $\log\theta_{n+1}\sim n^\sigma$, we conclude that the first series in the previous display is indeed finite. The second series is also finite since $f\circ h$ is regularly varying at infinity with index $-(\alpha-1)/\alpha$ (recall that $t_{n+1}=\phi(\theta_{n+1})$). 

Next we prove that $\sum_{n=1}^\infty (\exp(-\Psi_{f(t_n)}(\theta_n))-\exp(-\theta_n t_n))=\infty$. First observe that the terms $\exp(-\theta_nt_n)=\exp(-\log \log\theta_n)=n^{-\sigma}$ are summable. Applying Lemma~\ref{lem:asymp_equiv_Phi_domstable}(i) and~\eqref{eq:stable_f_asymp4}, we obtain $\Psi_{f(t_n)}(\theta_n)\sim (1-\rho)\log(\theta_n G^{-1}(f(t_n)))$ as $n \to \infty$. As in Part~(i), for all sufficiently large~$n$ we have 
$\Psi_{f(t_n)}(\theta_n)
\le (1+\ve)p(1-\rho)
    \log\log\theta_n$. Thus 
$\exp(-\Psi_{f(t_{n})}(\theta_n))
\ge n^{-\sigma(1+\ve)p(1-\rho)}$ and, since $\sigma(1+\ve)p(1-\rho)<1$, we deduce that $\sum_{n=1}^\infty \exp(-\Psi_{f(t_n)}(\theta_n))=\infty$, completing the proof.
\end{proof}

\subsection{Proofs of Subsection~\ref{subsec:applications}}
In this subsection we prove the results stated in Subsection~\ref{subsec:applications}.

\begin{proof}[Proofs of Lemmas~\ref{lem:upper_fun_Lev_path_post_slope} and~\ref{lem:upper_fun_Lev_path}]
We first prove Lemma~\ref{lem:upper_fun_Lev_path_post_slope}.
Let $s\in\mL^+(\mS)$ and let the function $f:[0,\infty) \to [0,\infty)$ be continuous and increasing with $f(0)=0$ and define the function $\tilde f(t)\coloneqq\int_0^t f(u)\D u$, $t\ge 0$. Note that $m_s= X_{\tau_s}\wedge X_{\tau_s-}$ equals $C_{\tau_s}$ since $\tau_s$ is a contact point between $t\mapsto X_t\wedge X_{t-}$ and its convex minorant $C$. 

Part (i). By assumption, for any $M>0$ there exists $\delta>0$ such that $C'_{t+\tau_s}-s \ge Mf(t)$ for $t\in (0,\delta)$. Since $\int_0^t (C'_{u+\tau_s}-s)\D u=C_{t+\tau_s}-m_s-st$ it follows that $C_{t+\tau_s}-m_s-st\ge M \tilde f(t)$ for all $t \in [0,\delta)$. Note that the path of $X$ stays above its convex minorant, implying $C_{t+\tau_s}-m_s-st\le X_{t+\tau_s}-m_s-st$. Thus, $X_{t+\tau_s}-m_s-st\ge M \tilde f(t)$ for all $t \in [0,\delta)$, implying that $\liminf_{t \da 0}(X_{t+\tau_s}-m_s-st)/\tilde f(t)\ge M$.

Part (ii). Assume that $\tilde f$ is convex on a neighborhood of $0$, and that $\limsup_{t\da 0} (C'_{t+\tau_s}-s)/f(t)=0$. Then, for all $M>0$ there exists some $\delta>0$ such that $C'_{t+\tau_s}-s\le M f(t)$ for all $t\in [0,\delta)$. Integrating this inequality gives $C_{t+\tau_s}-m_s-st\le M \tilde f(t)$ for all $t \in [0,\delta)$. Since $s\in\mL^+(\mS)$, there exists a decreasing sequence of slopes $s_n\da s$ such that $t_n=\tau_{s_n}-\tau_s\da 0$ and $X_{t_n+\tau_s}\wedge X_{t_n+\tau_s-}=C_{t_n+\tau_s}$ for all $n\in\N$. Thus, either $X_{t_n+\tau_s}-m_s-st_n \le M\tilde f(t_n)$ i.o. or $X_{t_n+\tau_s-}-m_s-st_n \le M\tilde f(t_n)$ i.o. Since $\tilde f$ is continuous, we deduce that  $\liminf_{t \da 0}(X_{t+\tau_s}-m_s-st)/\tilde f(t)\le M$. 

The proof of Lemma~\ref{lem:upper_fun_Lev_path} follows along similar lines with $\tilde f(t)=\int_0^t f(u)^{-1}\D u$, $t>0$, the slope $s=-\infty$ and $m_{-\infty}=X_0=0$.
\end{proof}

\begin{proof}[Proof of Corollary~\ref{cor:post-tau_s-Levy-path-attraction}]
Part (i) follows from Theorem~\ref{thm:upper_fun_C'_post_min} and Lemma~\ref{lem:upper_fun_Lev_path_post_slope}(ii). 

Part (ii). Assume $\alpha\in(1/2,1)$. By Theorem~\ref{thm:post-min-lower} and Lemma~\ref{lem:upper_fun_Lev_path_post_slope}(i) it suffices to prove that \eqref{eq:post-min-Pi-large}--\eqref{eq:post-min-Pi-mean} hold for $c=1$. As described in Subsection~\ref{subsec:simp_suff_cond_tau_s},  condition~\eqref{eq:simple_suff_cond_post_min_density} implies~\eqref{eq:post-min-Pi-var}--\eqref{eq:post-min-Pi-mean}. By Lemma~\ref{lem:generalized_Picard}, the density of $(X_t-st)/g(t)$ is uniformly bounded in $t>0$. Hence, the following condition implies~\eqref{eq:simple_suff_cond_post_min_density}:
\begin{equation}
\label{eq:cor_RV_suff1}
\int_0^1\int_{f(t/2)}^1 \frac{1}{f^{-1}(x)} \D x \frac{t}{g(t)}\D t<\infty.
\end{equation} 
Similarly,~\eqref{eq:post-min-Pi-large} holds with $c=1$ if $\int_0^1 (f(t)/g(t))\D t<\infty$.
Thus, it remains to show that~\eqref{eq:cor_RV_suff1} holds and 
$\int_0^1 (f(t)/g(t))\D t<\infty$.

We first establish~\eqref{eq:cor_RV_suff1}. Let $a=\alpha/(1-\alpha)$ and note that $f(t)\coloneqq 1/G(t(\log t^{-1})^p)=t^{1/a}\wt\varpi(t)$ where the slowly varying function $\wt\varpi$ is given by $\wt\varpi(t)=\log^{p/a}(1/t)\varpi(t\log^{p}(1/t))$. Thus, by~\cite[Thm~1.5.12]{MR1015093}, the inverse $f^{-1}$ of $f$ admits the representation $f^{-1}(t)=t^a \wh\varpi(t)$ for some slowly varying function $\wh\varpi(t)$. This slowly varying function satisfies 
\begin{equation}
\label{eq:hat_varpi}
t=f^{-1}(f(t))
=f(t)^a\wh\varpi(f(t))
\implies \wh\varpi(f(t))\sim t/f(t)^a\sim 1/\wt\varpi(t)^a, \qquad \text{ as } t \da 0.
\end{equation}

Since $a>1$, the function $f^{-1}$ is not integrable at $0$. Thus, by Karamata's theorem~\cite[Thm~1.5.11]{MR1015093} and~\eqref{eq:hat_varpi}, the inner integral in~\eqref{eq:cor_RV_suff1} satisfies 
\[
\int_{f(t/2)}^1\frac{1}{f^{-1}(x)}\D x
\sim \frac{1}{a-1}f(t/2)^{1-a}\wh\varpi(f(t))^{-1}
\sim \frac{2^{(a-1)/a}}{a-1}f(t)^{1-a}\wt\varpi(t)^a,
\qquad\text{as }t\da 0.
\]
Since $t/g(t)=t^{-1/a}/\varpi(t)$ for $t>0$, condition~\eqref{eq:cor_RV_suff1} holds if and only if the following integral is finite
\begin{equation*}
\int_0^1 f(t)^{1-a} \frac{\wt\varpi(t)^a}{\varpi(t)} t^{-1/a}\D t
=\int_0^1 \log^{p/a}(1/t)\frac{\varpi(t\log^{p}(1/t))}{\varpi(t)}\frac{\D t}{t}.
\end{equation*}
The integrand is asymptotically equivalent to $\log^{p/a}(1/t)$ since $\varpi(t\log^{p}(1/t))/\varpi(t) \to 1$ as $t \da 0$ uniformly on $[0,1]$ by~\cite[Thm~2.3.1]{MR1015093} and our assumption on $\varpi$. Thus, the condition $p<-a$ makes the integral in display finite, proving condition~\eqref{eq:cor_RV_suff1}.

To prove that $\int_0^1 (f(t)/g(t))\D t<\infty$, take any $\delta>0$ with $p(1/a-\delta)<-1$ (recall $p/a<-1$ by assumption) and apply Potter's bound \cite[Thm~1.5.6(iii)]{MR1015093} with $\delta$ to obtain, for some constant $K>0$,
\[
\int_0^1 \frac{f(t)}{g(t)}\D t=\int_0^1 \frac{g(t\log^p(1/t))}{g(t)\log^p(1/t)}\frac{\D t}{t} \le K\int_0^1 \log^{p(1/a-\delta)}(1/t)\frac{\D t}{t}<\infty.
\]

Part (iii). The result follows from  Corollary~\ref{cor:power_func_liminf_post_min} and Lemma~\ref{lem:upper_fun_Lev_path_post_slope}(i). 
\end{proof}

\section{Concluding remarks}
\label{sec:concluding_rem}

The points on the boundary of the convex hull of a L\'evy path where the slope increases continuously were characterised (in terms of the law of the process) in our recent paper~\cite{SmoothCM}. In this paper we address the question of the rate of increase for the derivative of the boundary at these points in terms of lower and upper functions, both when the tangent has finite slope and when it is vertical (i.e. of infinite slope). Our results cover a large class of L\'evy processes, presenting a comprehensive picture of this behaviour. Our aim was not to provide the best possible result in each case and indeed many extensions and refinements are possible. Below we list a few that arose while discussing our results in Section~\ref{sec:small-time-derivative} as well as other natural questions.

\begin{itemize}[leftmargin=2em, nosep]
\item Find an explicit description of the lower (resp. upper) fluctuations in the finite (resp. infinite) slope regime for L\'evy processes in the domain of attraction of an $\alpha$-stable process in terms of the normalising function (cf. Corollaries~\ref{cor:power_func_liminf_post_min} and~\ref{cor:stable_limsup}). In the finite slope regime, this appears to require a refinement of~\cite[Thm~4.3]{picard_1997} for processes in this class. 
\item In Theorems~\ref{thm:upper_fun_C'_post_min} and~\ref{thm:lower_fun_C'} we find the correct power of the logarithmic factor, in terms of the positivity parameter $\rho$, in the definition of the function $f$ for processes in the domain of attraction of an $\alpha$-stable process. It is natural to ask what powers of iterated logarithm arise and how the boundary value is linked to the characteristics of the L\'evy process. This question might be tractable for $\alpha$-stable  processes since power series and other formulae exist for their transition densities~\cite[Sec.~4]{MR1745764}, allowing higher order control of the Laplace transform $\Phi$ in Lemmas~\ref{lem:asymp_equiv_Psi_domstable_post_min} and~\ref{lem:asymp_equiv_Phi_domstable}.
\item Find the analogue of Theorems~\ref{thm:upper_fun_C'_post_min} and~\ref{thm:lower_fun_C'} for processes attracted to Cauchy process (see Remarks~\ref{rem:exclusions-tau}(a) and~\ref{rem:exclusions-0}(b) for details). 
\item Find L\'evy processes for which  there exists a deterministic function $f$ such that any of the following limits is positive and finite: $\limsup_{t\da 0}(C'_{t+\tau_s}-s)/f(t)$, $\liminf_{t\da 0}(C'_{t+\tau_s}-s)/f(t)$, $\limsup_{t\da 0}|C'_{t}|f(t)$ or $\liminf_{t\da 0}|C'_{t}|f(t)$.  By Corollaries~\ref{cor:power_func_liminf_post_min} and~\ref{cor:stable_limsup}, such a function does not exist for the limits $\liminf_{t\da 0}(C'_{t+\tau_s}-s)/f(t)$ or $\limsup_{t\da 0}|C'_{t}|f(t)$ within the class of regularly varying functions and $\alpha$-stable processes with jumps of both signs.
\end{itemize}

\printbibliography

\section*{Acknowledgements}

\thanks{
\noindent JGC and AM are supported by EPSRC grant EP/V009478/1 and The Alan Turing Institute under the EPSRC grant EP/N510129/1; 
AM was supported by the Turing Fellowship funded by the Programme on Data-Centric Engineering of Lloyd's Register Foundation; DB is funded by the CDT in Mathematics and Statistics at The University of Warwick. All three authors would like to thank the Isaac Newton Institute for Mathematical Sciences, Cambridge, for support and hospitality during the programme for Fractional Differential Equations where work on this paper was undertaken. This work was supported by EPSRC grant no EP/R014604/1.
}

\appendix

\section{Elementary estimates}

Recall that $(\gamma,\sigma^2,\nu)$ is the generating triplet of $X$ and the definition of the functions $\ov\gamma$, $\ov\sigma^2$ and $\ov\nu$ in~\eqref{eq:ov_functions} above.

\begin{lemma}
\label{lem:upper_tail_bound}
For any $p\in(0,2]$, $t,K>0$ and $\ve\in(0,1]$, the following bounds hold
\begin{align*}
\E[(|X_t|\wedge K)^p]
&\le (\ov\gamma(\ve)^2 t^2 + \ov\sigma^2(\ve)t)^{p/2} + K^p\ov\nu(\ve)t,\\
\p(|X_t|\ge K)
&\le (\ov\gamma(\ve)^2 t^2 + \ov\sigma^2(\ve)t)/K^2 + \ov\nu(\ve)t.
\end{align*}
\end{lemma}

\begin{proof}
Let $X_t=\ov\gamma(\ve)t + J_t + M_t$ be the L\'evy-It\^o decomposition of $X$ at level $\ve$, where $J$ is compound Poisson containing all of the jumps of $X$ with magnitude at least $\ve$ and $M_t$ is a martingale with jumps of size smaller than $\ve$. Fix $t>0$ and define the event $A$ 
of not observing any jump of $J$ on the time interval $[0,t]$. Clearly $1-\p(A)=1-e^{-\ov\nu(\ve)t}\le \ov\nu(\ve)t$. Consider the elementary inequality
$|X_t|^p\wedge K^p\le |\ov\gamma(\ve) t+M_t|^p\1_{A}+K^p\1_{A^c}$. Taking expectations 
and applying Jensen's inequality (with the concave function $x\mapsto x^{p/2}$ on $(0,\infty)$), we obtain the bound
\[
\E\big(|X_t|^p\wedge K^p\big)
\le \big(\ov\gamma(\ve)^2t^2 + \E\big[M_t^2\big]\big)^{p/2}
+ K^p(1-\p(A)),
\]
because $\E M_t=0$. The first inequality readily follows. The second inequality follows from the first one: using Markov's inequality we get
\[
\p(|X_t|\geq  K)
=\p(|X_t|\wedge K\geq  K)\leq \E(X_t^2\wedge K^2)/K^2.
\]
Thus, the second result follows from the first with $p=2$. 
\end{proof}

\end{document}